\documentclass[11pt]{amsart}
\usepackage[utf8]{inputenc}
\usepackage{amsmath,amssymb}
\usepackage{mathtools}
\usepackage[T1]{fontenc}
\usepackage{amsxtra,amsmath,amsfonts,amscd, amssymb, mathrsfs, amsthm, mathtools}

\usepackage[all]{xy}
\usepackage{enumerate}
\usepackage{enumitem}
\usepackage[hidelinks]{hyperref}
\usepackage{color}
\usepackage{comment}

\definecolor{orange}{rgb}{1,0.5,0}

\definecolor{brass}{rgb}{0.71, 0.65, 0.26}

\numberwithin{paragraph}{section}
\numberwithin{equation}{section}

\newcommand{\pshenvelope}{{P}}

\newcommand{\Acal}{\mathcal{A}}

\newcommand{\Bcal}{\mathcal{B}}

\newcommand{\Ccal}{\mathcal{C}}

\newcommand{\Dcal}{\mathcal{D}}

\newcommand{\Fcal}{\mathscr{F}}
\newcommand{\Gbb}{\mathbb{G}}

\newcommand{\Hcal}{\mathscr{H}}

\newcommand{\Jcal}{\mathcal{J}}

\newcommand{\Lcal}{\mathscr{L}}

\newcommand{\Ocal}{\mathcal{O}}
\newcommand{\Pbb}{\mathbb{P}}

\newcommand{\Rbb}{\mathbb{R}}

\newcommand{\Sbb}{\mathbb{S}}

\newcommand{\Tbb}{\mathbb{T}}

\newcommand{\Ucal}{\mathcal{U}}

\newcommand{\Xcal}{\mathscr{X}}

\newcommand{\Zcal}{\mathcal{Z}}

\newcommand{\fa}{\mathfrak{a}}
\newcommand{\fb}{\mathfrak{b}}

\newcommand{\C}{\mathbb{C}}

\newcommand{\N}{\mathbb{N}}

\newcommand{\Q}{\mathbb{Q}}

\newcommand{\R}{\mathbb{R}}

\newcommand{\Z}{\mathbb{Z}}

\newcommand{\Lan}{L^{\an}}

\DeclareMathOperator{\red}{{red}}
\DeclareMathOperator{\an}{{an}}
\DeclareMathOperator{\ev}{{ev}}
\DeclareMathOperator{\emb}{{emb}}
\DeclareMathOperator{\abs}{{abs}}
\DeclareMathOperator{\rint}{{relint}}
\DeclareMathOperator{\Val}{{val}}
\renewcommand{\and}{\operatorname{and}}

\DeclareMathOperator{\Div}{Div}

\DeclareMathOperator{\Id}{Id}

\DeclareMathOperator{\Spec}{{Spec}}

\DeclareMathOperator{\PSH}{PSH}
\DeclareMathOperator{\MA}{MA}

\DeclareMathOperator{\vol}{vol}
\newcommand{\metr}{{\|\phantom{a}\|}}
\newcommand{\absval}{{|\phantom{a}|}}

\renewcommand{\div}{{\operatorname{div}}}

\newcommand{\Xan}{{X^{\an}}}

\def\KO{{\mathcal O}}

\def\KL{{\mathscr L}}

\def\KX{{{\mathscr X}}}
\def\KY{{\mathscr Y}}
\def\KU{{\mathscr U}}
\def\an{{\mathrm{an}}}
\def\red{{\mathrm{red}}}
\def\metric{{\|\ \|}}
\def\metr{{\|\ \|}}
\def\modelD{{\mathscr D}}
\def\Pic{{{\rm Pic}\,}}

\def\kcirc{{K^\circ}}
\def\MA{{\rm MA}}
\def\PSH{{\rm PSH}}

\def\Da{{\mathfrak a}}
\def\Db{{\mathfrak b}}

\theoremstyle{plain}
\newtheorem{theo}{Theorem}[section]

\newtheorem{prop}[theo]{Proposition}

\newtheorem{lemma}[theo]{Lemma}
\newtheorem{cor}[theo]{Corollary}

\newtheorem{question}{Question}

\theoremstyle{definition}
\newtheorem{defi}[theo]{Definition}

\newtheorem{assumption}[theo]{Assumption}

\newtheorem{art}[theo]{}

\theoremstyle{remark}
\newtheorem{rem}[theo]{Remark}

\DeclareMathOperator{\Hom}{{Hom}}
\DeclareMathOperator{\rec}{{rec}}
\DeclareMathOperator{\cone}{{cone}}
\DeclareMathOperator{\conv}{{conv}}

\title[Continuity of Plurisubharmonic Envelopes] 
{Continuity of Plurisubharmonic Envelopes in 
Non-Archimedean Geometry {and Test Ideals}\\
(with an Appendix by
Jos\'e Ignacio Burgos Gil and Martín Sombra) }

 \author[W.~Gubler]{Walter Gubler}
\address{W. Gubler, Mathematik, Universit{\"a}t 
 Regensburg, 93040 Regensburg, Germany}
 \email{walter.gubler@mathematik.uni-regensburg.de}

 \author[P.~Jell]{Philipp Jell}
 \address{P. Jell,  
Georgia Institute of Technology,
686 Cherry Street,
Atlanta, GA 30332-0160}
\email{philipp.jell@math.gatech.edu}

 \author[K.~Künnemann]{Klaus K{\"u}nnemann}
 \address{K. K{\"u}nnemann, Mathematik, Universit{\"a}t 
 Regensburg, 93040 Regensburg, Germany}
 \email{klaus.kuennemann@mathematik.uni-regensburg.de}

 \author[F.~Martin]{Florent Martin}
 \address{F. Martin, Mathematik, Universit{\"a}t 
 Regensburg, 93040 Regensburg, Germany}
 \email{florent.martin@mathematik.uni-regensburg.de}

\date{\today}

 \thanks{W.~Gubler, 
K.~Künnemann and F.~Martin 
 were supported by the collaborative research 
 center SFB 1085 funded by the Deutsche Forschungsgemeinschaft. 
P.~Jell was for part of this work also supported by the SFB 1085 and for the other part by the DFG Research Fellowship JE 856/1-1. 
J.~I.~Burgos was partially supported by MINECO research projects MTM2016-79400-P and by ICMAT Severo Ochoa project SEV-2015-0554.
M.~Sombra was partially supported by the MINECO research project MTM2015-65361-P and through the "María de Maeztu" program for units of excellence MDM-2014-0445.}

\begin{document}

\begin{abstract}
Let $L$ be an ample line bundle on a smooth projective variety $X$ over a non-archimedean field $K$. 
For a continuous metric on $\Lan$, 
we show in the following two cases that the semipositive envelope is a continuous semipositive metric on $\Lan$ and that the non-archimedean Monge-Amp\`{e}re equation has a solution. 
First, we prove it for curves using results of Thuillier. 
Second, we show it under the assumption that $X$ is a surface defined geometrically over the function field of a curve over a perfect field $k$ of positive characteristic. 
The second case holds in higher dimensions if we assume resolution of singularities over $k$.   
The proof follows a strategy from Boucksom, Favre and Jonsson, replacing multiplier ideals by test ideals.
Finally,  the appendix by Burgos and Sombra provides an example of a semipositive metric whose retraction is not semipositive.
The example is based on the construction of
a toric variety which has 
two SNC-models which induce the same skeleton but different
retraction maps.
\bigskip

\end{abstract}

\date{\today}

\maketitle

\tableofcontents

\section{Introduction}

Let $L$ be an ample line bundle on an $n$-dimensional complex projective variety $X$ 
and let $\mu$ be a smooth volume form on the associated complex manifold $\Xan$ of total mass $\deg_L(X)$. 
The \emph{Calabi conjecture} claims that there is a smooth semipositive metric $\metr$ on $\Lan$, unique up to positive multiples,  solving the \emph{Monge--Amp\`ere equation}
\begin{equation}   \label{MA equation}
	c_1(L,\metr)^{\wedge n} = \mu.
\end{equation}
This was conjectured by Calabi who proved uniqueness \cite{Calabi-Proceedings, Calabi-incollection} and 
the existence part was solved by Yau \cite{yau78}. 
In fact, they proved a more general version in the setting of compact K\"ahler manifolds, but this will not be relevant for our paper. 

The motivation of this paper is the study of the non-archimedean version of this conjecture. 
We consider a non-archimedean field $K$ with valuation ring $\kcirc$. 
Let $L$ be an ample line bundle on an $n$-dimensional smooth projective variety $X$ over $K$. 
The line bundle $L$ induces a line bundle $L^\an$  
on  the analytification $\Xan$ of $X$ as a Berkovich 
non-archimedean analytic space. 
In non-archimedean  geometry, 
model metrics on $\Lan$ play a similar role as smooth metrics
on line bundles on complex manifolds. 
We call a model metric on $\Lan$ semipositive if it is induced by an nef model. 
Zhang {\cite{zhang1995b}} introduced continuous semipositive metrics on $\Lan$ as uniform limits of semipositive model metrics. 
For such metrics, Chambert-Loir {\cite{chambert-loir2006}} defined a Monge--Amp\`ere measure $c_1(L,\metr)^{\wedge n}$ on $\Xan$ 
which is a positive Radon measure of total mass $\deg_L(X)$.
These measures play an important role in arithmetic equidistribution results (see \cite{Yua08}). 
We refer to Section \ref{setup and goal}  for details about these notions.

In the \emph{non-archimedean Calabi--Yau problem}, one looks for continuous semipositive metrics on $\Lan$ 
solving the Monge--Amp\`ere equation \eqref{MA equation}. 
Yuan and Zhang \cite{yuan-zhang} proved that such a metric is unique up to constants. 
The existence of a singular semipositive solution was proven in the case of curves by Thuillier \cite[Cor.~3.4.13]{thuillier05:_theor}. 
Liu \cite{liu2011} proved existence of a continuous semipositive solution for totally degenerate abelian varieties $A$ if $\mu$ is a smooth volume form on the 
{canonical} skeleton of $A$. 

Next, we describe the fundamental existence result of Boucksom, Favre and Jonsson \cite{BFJ1, BFJ2}. 
We assume that $K$ is a complete discretely valued field with valuation ring $\kcirc$. 
We recall that an SNC-model is a regular model of $X$ such that the special fiber has simple normal crossing support. 
{Boucksom, Favre and Jonsson} prove in \cite[Thm.~A]{BFJ2} 
that the non-archimedean Calabi--Yau problem has a continuous semipositive solution $\metr$ if the following assumptions are satisfied:
\begin{itemize}
	\item[(a)] The characteristic of the residue field $\tilde{K}$ is zero.
	\item[(b)] The positive Radon measure $\mu$ is supported on the skeleton of a projective SNC-model of $X$ and satisfies $\mu(\Xan)=\deg_L(X)$. 
	\item[(c)] The smooth projective variety $X$ is of geometric origin from a $1$-dimensional family over $\tilde{K}$.
\end{itemize}	
The last condition will play an important role in this paper. 
More generally, we say that $X$ is \emph{of geometric origin from a $d$-dimensional family over the field $k$} 
if there is a codimension $1$ point $b$ in a normal variety $B$ over $k$ such that $\kcirc$ is the completion of $\Ocal_{B,b}$ and 
such that $X$ is defined over the function field $k(B)$. 
In \cite[Thm.~D]{BGJKM}, we have shown that 
{(c) is not necessary for the existence of a continuous semipositive solution of the non-archimedean Calabi--Yau problem if we assume (a) and (b).}

We will later look for a similar result in equicharacteristic $p >0$. 
To do so, we have to understand the basic ingredients in the proof 
of the existence result of Boucksom, Favre and Jonsson. 
In \cite{BFJ1}, {the authors} develop a global pluripotential theory on $\Xan$ for singular semipositive metrics 
using the piecewise linear structure on the skeletons of SNC-models. 
It is here where Assumption (a) enters the first time as resolution of singularities is used to have sufficiently many SNC-models of $X$ at hand. 
For a continuous metric $\metr$, we define the \emph{semipositive envelope} $\pshenvelope(\metr)$ by
\begin{equation}\label{envelope-metric}
\pshenvelope(\metr) \coloneqq \inf \{ \metr' \mid  \text{$\metr \leq \metr'$  {and $\metr'$ is a semipositive model metric on $\Lan$}} \}.
\end{equation}
It is absolutely crucial for pluripotential theory to prove that $\pshenvelope(\metr)$ is continuous as 
this is equivalent to the monotone regularization theorem (see \cite[Lemma 8.9]{BFJ1}). 
The monotone regularization theorem is the basis in \cite{BFJ2} to introduce 
the Monge--Amp\`ere measure, capacity and energy for singular semipositive metrics. 
The proof of continuity of $\pshenvelope(\metr)$ in \cite[\S 8]{BFJ1} uses multiplier ideals on regular projective models. 
In the proof of the required properties of multiplier ideals (see \cite[Appendix B]{BFJ1}), 
the authors use vanishing results which hold only in characteristic zero, 
and hence Assumption (a) plays an important role here as well. 

A second important result is the \emph{orthogonality property} for $\pshenvelope(\metr)$ given in \cite[Thm.~7.2]{BFJ2}. 
Multiplier ideals occur again in their proof and it is here where the geometric assumption (c) is used. 
However, it is shown in \cite[Thm.~6.3.3]{BGJKM} that continuity of $\pshenvelope(\metr)$ is enough 
to prove the orthogonality property without assuming (a) or (c). 
Then the variational method of Boucksom, Favre and Jonsson can be applied to prove existence of a {continuous semipositive} solution for the non-archimedean Calabi--Yau problem.

This makes it very clear that continuity of the semipositive envelope $\pshenvelope(\metr)$ plays a crucial role in the non-archimedean Calabi--Yau problem.
It is the main object of study in this paper. 
In Section \ref{setup and goal}, we will study the basic properties of a slight generalization of $\pshenvelope(\metr)$ 
which is called the $\theta$-psh envelope  for a closed $(1,1)$-form $\theta$ on $X$. 
For the sake of simplicity, we will restrict our attention in the introduction to the semipositive envelope $\pshenvelope(\metr)$, 
while all the results hold more generally for the $\theta$-psh envelope assuming that the de Rham class of $\theta$ is ample.

In Section \ref{section3}, we will look at continuity of the semipositive envelope in the case of a smooth projective curve $X$ over an
{arbitrary complete} non-archimedean field $K$. 
Potential theory on the curve $\Xan$ was developed in Thuillier's thesis \cite{thuillier05:_theor}.  
We will use Thuillier's results and the slope formula of Katz, Rabinoff, and Zureick-Brown \cite[Thm.~2.6]{katz-rabinoff-zureick} to prove:

\begin{theo} \label{continuity of semipositive envelope}
Let $L$ be an ample line bundle on a smooth projective curve $X$ over any non-archimedean field $K$. 
Then $\pshenvelope(\metr)$ is a continuous semipositive metric on $\Lan$ for any continuous metric $\metr$ on $\Lan$.
\end{theo}

A slightly more general version will be proved in Theorem \ref{thm-continuity-theta-envelope-curves}. 
The following is important in the proof: 
Let $L$ be any line bundle on the smooth projective curve $X$. 
We assume that $X$ has a strictly semistable model $\Xcal$ such that  $L^\an$ has a model metric $\metr_0$ associated to a line bundle on $\Xcal$. 
For any metric $\metr$ on $\Lan$, we consider the function $\varphi \coloneqq - \log(\metr/\metr_0)$. 
Let $p_\Xcal \colon \Xan \to \Delta$ be the canonical retraction to the skeleton $\Delta$ associated to $\Xcal$. Then
\begin{equation} \label{retraction and semipositivity}
 \metr_{\Delta} \coloneqq e^{-\varphi \circ p_\Xcal} \metr_0
\end{equation}
is a metric on $\Lan$ which does not depend on the choice of $\metr_0$. 
The following result is crucial in the proof of Theorem \ref{continuity of semipositive envelope}:

\begin{prop} \label{restriction of the metric to skeleton}
	Using the hypotheses above, we consider a model metric $\metr$ of $L^\an$. Then we have the following properties:
\begin{itemize}
	\item[(i)] {The metric} $\metr_\Delta$ is a model metric.
	\item[(ii)] {There is an equality of measures}
	$c_1(L, \metr_\Delta) = (p_\Xcal)_*(c_1(L,\metr)).$
	\item[(iii)] If $\metr$ is semipositive, then $\metr_\Delta$ is semipositive and $\metr_\Delta \leq \metr$.
\end{itemize}
\end{prop}

 This will be proven in Propositions \ref{model functions and skeleton} and \ref{lemma psh retraction curve}.   
It would make pluripotential theory and the solution of the non-archimedean Calabi--Yau problem much easier 
if Proposition \ref{restriction of the metric to skeleton} would also hold in higher dimensions as we could work more combinatorically on skeletons. 
It is still true that  $\metr_\Delta$ is a model metric {which satisfies} $\metr_\Delta\leq \metr$.  
Burgos and Sombra show in a two dimensional toric counterexample 
in the Appendix that $\metr_\Delta$  does not have to be semipositive. 
 
 We show now that Proposition \ref{restriction of the metric to skeleton} is also crucial for the existence of the solution 
of the non-archimedean Calabi--Yau problem in the case of curves arguing as in \cite[\S 9]{BFJ-simons}. 
By Thuillier \cite[Cor.~3.4.13]{thuillier05:_theor}, there is a semipositive metric $\metr$ solving \eqref{MA equation}, but it might be singular. 
Here, semipositive means that the metric is an increasing pointwise limit of semipositive model metrics of the ample line bundle $L$. 
 If we assume that the positive Radon measure $\mu$ has support in the skeleton $\Delta$ of a strictly semistable model $\Xcal$ of $X$, 
 then it follows easily from Proposition \ref{restriction of the metric to skeleton} 
that $\metr_\Delta$ is a continuous semipositive metric solving \eqref{MA equation}. 
 Burgos and Sombra show in their counterexample in the Appendix that this does not hold in higher dimensions either. 

To look for solutions of the higher dimensional non-archimedean Monge--Amp\`ere equation in positive characteristic, 
we will replace the use of multiplier ideals by the use of test ideals. 
Test ideals were introduced by Hara and Yoshida  {\cite{hara-yoshida2003}} using a generalization of tight closure theory.
In Section \ref{app-test-ideals}, we will gather the necessary facts about test ideals mainly following \cite{mustata13} and 
so we work on a smooth variety $X$ over a perfect field $k$ of characteristic $p>0$. 
Similarly as in the case of multiplier ideals, one can define an 
\emph{asymptotic test ideal $\tau(\lambda\|D\|)$ of exponent  $\lambda \in \R_{\geq 0}$} for a divisor $D$ on $X$. 
Crucial for us is that $\tau(\lambda\|D\|)$ satisfies a subadditivity property and the following \emph{uniform generation property}:

\begin{theo} \label{intro: uniform gen prop}
Let $X$ be a projective scheme over a finitely generated $k$-algebra $R$ such that $X$ is a smooth $n$-dimensional variety over $k$. 
We assume that $H$ is an ample and basepoint-free divisor, $D$ is a divisor with $h^0(X, \Ocal(mD)) \neq 0$ for some $m \in \N_{>0}$ and  
$E$ is a divisor   such that the $\Q$-divisor $D- \lambda E$ is nef for some $\lambda \in \Q_{\geq 0}$. 
Then the sheaf $\Ocal_X(K_{X/k}+E+dH) \otimes_{\Ocal_X} \tau(\lambda \cdot \|D\|)$ 
is globally generated for all $d \geq n+1$.
\end{theo} 

This was proven by Mustaţă if $X$ is projective over $k$. 
As we will later work over discrete valuation rings, we need the more general version with $X$ only projective over $R$. 
This will be possible in Theorem \ref{uniform generation} 
as we can replace the use of Fujita's vanishing theorem in Mustaţă's proof by Keeler's generalization. 

Now we come to the \emph{non-archimedean Calabi--Yau problem in equicharacteristic} $p>0$. 
For the remaining part of the introduction, we now fix  an $n$-dimensional smooth projective variety $X$ 
over a complete discretely valued field $K$ of characteristic $p$. 
To apply the results on test ideals, we have to require that $X$ is of geometric origin from a $d$-dimensional family over a perfect field $k$. 
We also fix an ample line bundle $L$ on $X$.

\begin{theo} \label{continuity of semipositive envelope2}
Under the hypotheses above, we assume that resolution of singularities holds over $k$ in dimension $d+n$. 
Then the semipositive envelope $\pshenvelope(\metr)$ of a continuous metric $\metr$ on $\Lan$ is a continuous semipositive metric on $\Lan$.
\end{theo} 

For the precise definition about resolution of singularities, we refer to Definition \ref{resolution of singularities}. 
As resolution of singularities is known in dimension $3$ over a perfect field 
by a result of Cossart and Piltant \cite[Thm.~p.~1839]{cossart-piltant2009}, 
Theorem \ref{continuity of semipositive envelope2} is unconditional 
if $X$ is a smooth projective surface of geometric origin from a $1$-dimensional family over $k$.

In Section \ref{section5}, we will prove Theorem \ref{continuity of semipositive envelope2} 
in the case when $\metr$ is a model metric associated to a model which is also defined geometrically over $k$. 
We will follow the proof of Boucksom, Favre, and Jonsson, replacing multiplier ideals by test ideals. 
As we use a rather weak notion of resolution of singularities, a rather subtle point in the argument is necessary 
in the proof of Lemma \ref{higher geometric model} which involves a result of P\'epin about semi-factorial models.  
Theorem \ref{continuity of semipositive envelope2} will be proved in full generality in Section \ref{section6} 
using the $dd^c$-lemma and basic properties of the semipositive envelope. 
In fact, we will prove in Theorem \ref{thm-continuity-theta-envelope} a slightly more general result.

If we use additionally that embedded resolution of singularities (see Definition \ref{embedded resolution of singularities}) 
holds over $k$ in dimension $d+n$, then the family of projective SNC-models will be cofinal in the category of all models of $X$. 
We will see in Section \ref{section7} that this and Theorem \ref{continuity of semipositive envelope2} allow us 
to set up the pluripotential theory from \cite{BFJ1} on $\Xan$. 
By \cite[Thm.~7.2]{BFJ2} again, the continuity of $\pshenvelope(\metr)$ yields 
that the orthogonality property holds for any continuous metric on $\Lan$. 
We will use this in Section \ref{section7} to show that the variational method of Boucksom, Favre and Jonsson 
proves the following result (see Theorem \ref{MA problem}). 

\begin{theo} \label{Calabi-Yau in equichar p}
Let $X$ be an $n$-dimensional smooth projective variety of geometric origin from a $d$-dimensional family over a perfect field $k$ of characteristic $p>0$. 
We assume that resolution of singularities and embedded resolution of singularities hold over $k$ in dimension $d+n$. 
Let $L$ be an ample line bundle on $X$ and 
let $\mu$ be a positive Radon measure supported on the skeleton of a projective SNC-model of $X$ with $\mu(\Xan)=\deg_L(X)$. 
Then the non-archimedean Monge--Amp\`ere equation \eqref{MA equation} has a continuous semipositive metric $\metr$ on $\Lan$ as a solution. 
\end{theo}

Cossart and Piltant \cite{cossart-piltant2008, cossart-piltant2009} have shown resolution of singularities  and embedded resolution of singularities 
in dimension $3$ over a perfect field, hence Theorem \ref{Calabi-Yau in equichar p} holds unconditionally 
for a smooth projective surface $X$ of geometric origin from a $1$-dimensional family over the perfect field $k$.

\subsection*{Acknowledgement}
We  thank Mattias Jonsson for his valuable comments on a first version, 
Matthias Nickel for helpful discussions, and 
the referee for his useful remarks.

\addtocontents{toc}{\protect\setcounter{tocdepth}{0}}
\section*{Notations and conventions}
\addtocontents{toc}{\protect\setcounter{tocdepth}{1}}

Let $X$ be a scheme. 
An \emph{ideal} in $\KO_X$ is a quasi-coherent
ideal sheaf in $\KO_X$.
A \emph{divisor} on $X$ is always a Cartier divisor on $X$.
{Given $m\in \N$ we write 
$X^{(m)}$ for the set of all $p\in X$ where the local ring
$\Ocal_{X,p}$ has Krull dimension $m$.}
Let $k$ be a field. A \emph{variety $X$ over $k$} is an integral $k$-scheme $X$
which is separated and of finite type. 
A \emph{curve (resp. surface)} is a variety of dimension one 
(resp. two).

Throughout this paper $(K,\absval)$ denotes a complete
non-archimedean valued field with valuation ring $K^\circ$ and
residue field $\tilde K$.
Starting in Section \ref{section5} we will assume furthermore 
that the valuation is discrete
and that $K$ has positive characteristic $p>0$.
In this case there exists an isomorphism
$K^\circ \stackrel{\sim}{\longrightarrow}\tilde K[[T]]$
\cite[Thm.~29.7]{matsumura-book}.
Let $X$ be a $K$-variety.
We denote the analytification of $X$ 
in the sense of Berkovich \cite[Thm.~3.4.1]{berkovich-book} by $X^\an$.

\section{Model metrics, semipositive metrics, and envelopes} 
\label{setup and goal}\label{section2}

Let $X$ be a proper variety over a complete non-archimedean
valued field $(K,\absval)$.

\begin{art} \label{models}
A \emph{model of $X$} is given by a proper flat scheme $\KX$ 
over $S:={\rm Spec}\,K^\circ$
together with an isomorphism $h$ between 
$X$ and the generic fiber $\KX_\eta$ of the $S$-scheme $\KX$
which we read as an identification.
Given a model $\KX$ of $X$ 
there is a canonical surjective \emph{reduction map}
$\red\colon X^\an\longrightarrow \KX_s$ where
$\KX_s$ denotes the special fiber 
$\KX\otimes_{K^\circ} \tilde K$ of $\KX$ over $S$.

Let $L$ be a line bundle on the proper variety $X$.
A \emph{model of $(X,L)$}
or briefly a \emph{model of $L$} is given by a model
$(\KX,h)$ of $X$ together with a line bundle $\KL$ on
$\KX$ and an isomorphism $h'$ between $L$ and $h^*(\KL|_{\Xcal_\eta})$
which we read  as an identification.

Given a model $(\KX,\KL)$ of $(X,L^{\otimes m})$
for some $m\in \N_{>0}$
there is a unique metric $\metric_\KL$
on $L^\an$ over $X^\an$ which satisfies the following:
Given an open subset $\KU$ of $\KX$,
a frame $t$ of $\KL$ over $\KU$, and a section 
$s$ of $L$ over $U=X\cap \KU$ we write
$s^{\otimes m}=ht$ for some regular function $h$ on $U$
and get 
$\|s\|=\sqrt[m]{|h|}$ 
on $U^\an\cap \red^{-1}(\KU_s)$.
{Such a metric on $L^\an$  is called a 
\emph{model metric determined on $\KX$}.}
\end{art}

\begin{art}  \label{model functions}
A model metric $\metric$ on  $\KO_{X^\an}$ induces a
continuous function $f=-\log \|1\|\colon X^\an\to \R$.
The space of \emph{model functions}
\[
\modelD(X)=\{f\colon X^\an\rightarrow \R\,|\,
f=-\log \|1\| 
\mbox{ for some model metric }\metric \mbox{ on }\KO_{X^\an}\}
\]
has a natural structure of a $\Q$-vector space.
We say that a model function $f=-\log \|1\|$ is 
\emph{determined on a model $\KX$} if the model metric 
$\metric$ is determined on $\KX$.
A vertical divisor $D$ on $\KX$ determines a model
$\KO(D)$ of $\KO_X$ and an associated model function
$\varphi_D  \coloneqq -\log\|1\|_{\KO(D)}$. 
Such model functions are called \emph{$\Z$-model functions}. 
Let $\Da$ denote a vertical ideal of $\KX$.
Let $E$ denote the exeptional divisor of the blowup of $\KX$ in $\Da$.
Then $\log|\Da|:=\varphi_E$ is called the \emph{{$\Z$-}model function defined
by the vertical ideal $\Da$}.
\end{art}

\begin{art} \label{form curvature etc}
Consider a model  $\KX$ of the proper variety $X$ over $K$.
The  rational vector space space $N^1(\KX/S)_\Q$ 
is by definition the quotient of
$\Pic(\KX)_\Q \coloneqq  {\rm Pic}(\KX) \otimes_\Z \Q$ by the subspace 
generated by classes of line bundles $\KL$ such that $\KL\cdot C=0$ 
for each closed curve $C$ in the special fiber $\KX_s$. 
{Note that $N^1(\KX/S)_\Q$ is finite dimensional by applying \cite[Prop.~IV.1.4]{kleiman-1966}
to $\KX_s$.} 
We define
$N^1(\KX/S)\coloneqq N^1(\KX/S)_\Q\otimes_\Q\R$. 
An element $\alpha \in N^1(\KX/S)_\Q$ (resp.~$\alpha \in N^1(\KX/S))$ 
is called \emph{nef} if $\alpha \cdot C \geq 0$ 
for all closed curves $C$ in $\KX_s$.  
We call a line bundle $\KL$ on $\KX$ \emph{nef} if 
the class of $\KL$ in $N^1(\KX/S)$ is nef.
\end{art}

\begin{art}
We define $\Zcal^{1,1}(X)_\Q$ as the direct limit
\begin{align} \label{form definition}
\Zcal^{1,1}(X)_\Q  \coloneqq \varinjlim N^1(\KX / S)_\Q,
\end{align}
where $\KX$ runs over the isomorphism classes of models of $X$.
The \textit{space of closed $(1,1)$-forms on $X$} is defined as 
$\Zcal^{1,1}(X)  \coloneqq \Zcal^{1,1}(X)_\Q\otimes_\Q\R$.
Let $L$ be a line bundle on $X$.
Let $\metr$ be a model metric on $L^\an$ which is determined on $\KX$ by
a model $\KL$ of $L^{\otimes m}$.
We multiply the class of $\KL$ in $N^1(\KX/S)_\Q$ by $m^{-1}
$ which determines a well defined class
$c_1(L,\metr)\in \Zcal^{1,1}(X)_\Q\subseteq \Zcal^{1,1}(X)$ 
{called} the \textit{curvature form $c_1(L,\metr)$ of 
$(L, \metr)$}.
{We have a natural map $dd^c \colon \modelD(X) \to \Zcal^{1,1}(X); \;\; f \mapsto c_1(\Ocal_X, \metr_{\rm triv} \cdot e^{-f})$.}

A closed $(1,1)$-form $\theta$ is called {\it semipositive} if it is represented by a nef
element $\theta_\KX \in N^1(\KX/S)$ for some model $\KX$ of $X$. 
We say that a model metric $\metr$ on $L^\an$ for
a line bundle $L$ on $X$ is {\it semipositive} 
if the same holds for the curvature form 
$c_1(L,\metr)$. 
\end{art}

\begin{art}\label{max psh}
For $\theta \in \Zcal^{1,1}(X)$ we denote by
\[
\PSH_\modelD(X, \theta) = \{ f \in \modelD(X) \, \vert \, \theta + 
{dd^c f}
\in \Zcal^{1,1}(X) \text{ is semipositive} \}
\]
the set of \emph{$\theta$-plurisubharmonic} 
($\theta$-psh for short) model functions. 
Recall from  
\cite[Prop.~3.12]{gubler-martin} that
the set $\PSH_\modelD(X, \theta)$ is stable under the formation
of max.
\end{art}

\begin{art} \label{Neron Severi-1}
If  $Y$ is a proper variety over  an arbitrary field $k$, 
we denote by $N^1(Y)_\Q$ the rational vector space $\Pic(Y) \otimes \Q$ 
modulo numerical equivalence. 
Similarly, we denote by $N^1(Y)=N^1(Y)_\Q\otimes_\Q\R$ 
the real vector space 
$\Pic(Y) \otimes \R$ modulo numerical equivalence. 
A class in $N^1(Y)$ is called {\it ample} if it is an    
$\R_{>0}$-linear combination of classes induced by ample line bundles on $Y$.
{An element $\alpha \in N^1(Y)_\Q$ (resp.~$\alpha \in N^1(Y))$ 
is called \emph{nef} if $\alpha \cdot C \geq 0$ 
for all closed curves $C$ in $Y$. }
\end{art}

\begin{art} \label{Neron Severi-2}
The restriction maps 
$N^1(\KX/S) \rightarrow N^1(X),\,[\KL]\mapsto[\KL|_X]$ induce a linear map
$\{\phantom{a}\}: \Zcal^{1,1}(X) \longrightarrow  N^1(X), \, \theta \mapsto \{\theta \}$.
We call $\{\theta\}$ the \emph{de Rham class} of $\theta$.
\end{art}

\begin{defi} \label{definition envelope}
Let $X$ be a projective variety over $K$ and 
$\theta\in \Zcal^{1,1}(X)$ with de Rham class $\{\theta\}\in N^1(X)$.
The \emph{$\theta$-psh envelope {$P_\theta(u)$}} of 
{$u\in C^0(X^\an)$} {is the function}
\begin{equation}\label{def-psh-envelope}
{\pshenvelope_\theta(u) \colon X^\an\to \R\cup\{-\infty\},\,\,}
\pshenvelope_\theta(u)(x)=\sup \{
\varphi(x)\,|\,\varphi \in 
\PSH_\modelD(X,\theta)\wedge \varphi\leq u\}.
\end{equation}
\end{defi}

Note that $\pshenvelope_\theta(u)$ is a real valued function if and only if there exists a  $\theta$-psh model function. 
For the existence of a $\theta$-psh model function, 
it is necessary that the de Rham class $\{\theta\}$ is nef (see \cite[4.8]{gubler-martin} and \cite[Rem.~5.4]{BFJ1}). 
If $\{\theta\}$ is ample, then there exists always a $\theta$-psh model function and
hence $\pshenvelope_\theta(u)$ is a real valued function. 
If there is no $\theta$-psh function, then $\pshenvelope_\theta(u) \equiv - \infty$ by definition.
	
If the residue characteristic is zero and if the de Rham class $\{\theta\}$ is ample,
our definition of $\pshenvelope_\theta(u)$ is 
by \cite[Thm.~8.3 and Lemma 8.9]{BFJ1} equivalent
to the definition of Boucksom, Favre, and Jonsson
in \cite[Def.~8.1]{BFJ1} .

The next proposition collects elementary properties of envelopes.

\begin{prop}\label{regulari-prop2} 
Let $u,u'\in C^0(X^\an)$ and $\theta, \theta'\in \Zcal^{1,1}(X)$. 
\begin{enumerate}
\item
If $u\leq u'$ then $\pshenvelope_\theta(u)\leq \pshenvelope_\theta(u')$.
\item
We have
$\pshenvelope_{t\theta+(1-t)\theta'}(tu+(1-t)u')
\geq t\pshenvelope_\theta(u)+(1-t)\pshenvelope_{\theta'}(u')$
for all $t\in [0,1]$.
\item
We have
$\pshenvelope_\theta(u)+c=\pshenvelope_\theta(u+c)$ for each $c\in\R$.
\item
We have 
$\pshenvelope_\theta(u)-v=\pshenvelope_{\theta+dd^cv}(u-v)$ 
for each $v\in \modelD(X)$.
\item
{If $\pshenvelope_\theta(u) \not\equiv - \infty$, then we have} 
$\sup_{X^\an} |\pshenvelope_\theta(u)-\pshenvelope_\theta(u')|
\leq \sup_{X^\an} |u-u'|$.
\item
If $\theta$ is determined on a model $\KX$, 
{if the de Rham class $\{\theta\}$ is ample and if}  $\theta_m\to \theta$
in $N^1(\KX/S)$, then $P_{\theta_m}(u)\to \pshenvelope_\theta(u)$ uniformly 
on $X^\an$.
 \item 
We have $\pshenvelope_{t \theta}(tu)= t \pshenvelope_\theta(u)$
for all $t\in \R_{>0}$.
\item
{Assume $\pshenvelope_\theta(u) \not\equiv - \infty$. 
Then} 
the envelope $P_\theta(u)$ is continuous if and only if it is a uniform limit of $\theta$-psh model functions.
\end{enumerate}
\end{prop}

\begin{proof}
The proof of Properties (i)--(vi) in \cite[Prop.~8.2]{BFJ1} 
works in our setup as well. 
For (vi) it was used that an ample line bundle extends to an ample line bundle on a sufficiently high model which holds in our more general setting 
by \cite[Proposition 4.11]{gubler-martin}.
{Property (vii) is obvious for $t\in\Q_{>0}$ and an easy approximation argument then shows (vii) in general.}
{We have seen that $\theta$-psh model functions are closed under $\max$ and hence the $\theta$-psh model functions $\varphi \leq u$ form a directed family. 
We conclude that} (viii) follows from Dini's Theorem for nets 
\cite[p.~239]{Kelley1975} and the definition of $P_\theta(u)$.
\end{proof}

For the next proposition, we assume for simplicity that the valuation on $K$ is discrete.

\begin{prop}\label{regulari-prop5}
Let $L$ be an ample line bundle on $X$, $\KL$ an extension to a
model $\KX$ and $\theta={c_1(L,\metr_\Lcal)}\in \Zcal^{1,1}(X)$.
For $m>0$ let
\begin{equation}\label{def-base-ideal-rel}
\Da_m=\mbox{\rm Im}\, \bigl(H^0(\KX,\KL^{\otimes m})\otimes_{K^\circ} 
\KL^{\otimes -m}\to \KO_\KX\bigr)
\end{equation}
be the $m$-th base ideal of $\KL$ and 
$\varphi_m:=m^{-1}\log |\Da_m|$.
Then $\varphi_m\in \PSH_\modelD(X, \theta)$ and
\begin{equation}\label{regulari-prop5-eq}
\lim_{m\to\infty}\varphi_m=\sup_{m\in\N}\varphi_m=\pshenvelope_\theta(0)
\end{equation}
pointwise on $X^\an$.
\end{prop}

\proof
This is shown as in Step 1 of the proof of \cite[Thm.~8.5]{BFJ1}.
Observe that the arguments which show 
\eqref{regulari-prop5-eq} in \emph{loc.~cit.}~on the subset of 
quasimonomial points give us \eqref{regulari-prop5-eq}
immediately on $X^\an$ using our different definition
of the $\theta$-psh envelope.
\qed

\begin{prop}
\label{lemma A4 BFJ2}
Let $K'/K$ be a finite normal extension and let $q \colon X' := X \otimes_K{K'} \to X$ be the natural projection. 
For $\theta \in \Zcal^{1,1}(X)$  and {$u \in C^0(\Xan)$}, 
we have
\begin{equation}
q^*(P_\theta(u)) = P_{q^*\theta}(q^*(u)).
\end{equation}
\end{prop}

\begin{proof}
Splitting the extension $K' / K$ into a purely 
inseparable part and a Galois part, we can reduce to two cases.
In the first case of a purely inseparable extension,  
the result follows from Lemma \ref{lemma purely inseparable} below.
In the second case of a Galois extension, we can apply the 
argument of \cite[Lemma A.4]{BFJ2}.
\end{proof}

\begin{rem} \label{model functions Q linear functions}
{(i) We always equip strictly $K$-analytic spaces with the 
\emph{$G$-topology} induced 
by the strictly $K$-affinoid domains.}
{We refer to \cite[\S 2.2]{berkovich-book} for the notion of a strictly $K$-affinoid domain and to \cite[\S 1.6]{Berko93} for 
the construction of the $G$-topology.}

(ii)
{We recall from \cite[Def.~2.8, 2.11]{gubler-martin}  
that a piecewise $\Q$-linear function
on a strictly $K$-analytic space $W$ is a function $f \colon W \to \R$ 
such that there is a $G$-covering $\{U_i\}_{i \in I}$ of $W$ by strictly affinoid domains, 
analytic functions $\gamma_i \in \Ocal(U_i)^\times$ and non-zero $m_i \in \N$ 
with   $m_i f = -\log|\gamma_i|$  on $U_i$ for every $i \in I$.
By \cite[Rem.~2.6, Prop.~2.10]{gubler-martin}, the notions of model functions {and} $\Q$-linear functions agree. }
\end{rem}

\begin{lemma}
\label{lemma purely inseparable}
Let $L$ be a line bundle on $X$.
Let $K'/K$ be a finite purely inseparable extension and let $q \colon X' \coloneqq X \otimes_K{K'} \to X$ be the natural projection.
$G$-topology induced by the strictly $K$-affinoid domains. 
{Then the map $(X')^{\rm an}\to \Xan$ 
is a homeomorphism and  
 it also identifies the $G$-topologies.}  
The map $q^*$ induces a bijection  between the set of model metrics on $L$ and the set  of model metrics on 
$q^*(L)$.
Moreover this bijection identifies semipositive metrics on $L$ and on $q^*(L)$.
\end{lemma}

\begin{proof}
{As in Remark \ref{model functions Q linear functions}
 we use} the $G$-topology induced by the strictly $K$-affinoid domains.
{We first prove} that the map $q \colon (X')^{\rm an}\to \Xan$ 
is a homeomorphism and 
that it also identifies the $G$-topologies. 
In fact, this follows easily from the following claim:

\emph{Step 1: Let $V$ be a strictly affinoid space over $K$ and $V' \coloneqq V \hat{\otimes}_K K'$. 
Then the natural projection $q \colon V' \to V$ is a homeomorphism which identifies the $G$-topologies.}

Let $p^e= [K':K]$ be the degree of the purely inseparable field extension. It is clear that for every $g\in \Ocal(V')$, there is $f\in \Ocal(V)$ with 
\begin{equation} \label{inseparable degree}
g^{p^e}=f\circ q.
\end{equation}
This property easily shows that $q \colon V' \to V$ is a homeomorphism which we read now as an identification. 
Using that \eqref{inseparable degree} holds also for rational functions $g$ on $V'$ and $f$ on $V$, we see that $V$ and $V'$ have the same strictly rational domains. 
By the Gerritzen--Grauert theorem \cite[Cor.~7.3.5/3]{bosch-guentzer-remmert}, we deduce the Step 1.

Next we prove the bijective correspondence between the model metrics on $L$ and on $L'$. 
Since $L$ admits a model metric \cite[2.1]{gubler-martin},
it is enough to show that we have a bijective correspondence between model functions on $\Xan$ and 
model functions on $(X')^{\rm an}$. 
{By Remark \ref{model functions Q linear functions} we may check the same correspondence 
for $\Q$-linear functions. This may then be done $G$-locally {and hence the correspondence follows from the following step}.}

\emph{Step 2: Using the same assumptions as in Step 1, 
the map $f \mapsto f \circ q$  is an isomorphism from the group of {piecewise $\Q$-linear} functions on $V$ onto the group of {piecewise $\Q$-linear} functions on $V'$.}

Using the above definition of piecewise $\Q$-linear functions, Step 1 and \eqref{inseparable degree} yield easily Step 2.

To deduce the lemma, it remains to check that  the identification between the model metrics on $L$ and $L'$ preserves semipositivity. 
This is an easy consequence of the projection formula
applied to finite morphisms between closed curves in the special fibers of models. 
\end{proof}

\section{Continuity of plurisubharmonic envelopes on curves}\label{section3}

In this section, $K$ is any field endowed with a non-trivial non-archimedean complete valuation {$v\colon K\to \R$} with value group $\Gamma \subset \R$. 
In this section, we consider a smooth projective curve $X$ over $K$. 
The goal is to prove the following result:

\begin{theo}\label{thm-continuity-theta-envelope-curves}
If $\theta$ is a closed $(1,1)$-form on $\Xan$ 
with nef de Rham class $\{\theta\}$ and if $u \in C^0(\Xan)$, then $\pshenvelope_\theta(u)$ is 
a uniform limit of $\theta$-psh model functions and 
thus $\pshenvelope_\theta(u)$ is continuous on $\Xan$. 
\end{theo}

\begin{art} \label{strictly semistable}
As a main tool in the proof, we need strictly semistable models of $X$ and their canonical skeletons. 
This construction is due to Berkovich in \cite{berkovich-1999}. 
We recall here only the case of a smooth projective curve $X$ over $K$ for which we can also refer to \cite{thuillier05:_theor}. 

A  $\kcirc$-model 
$\Xcal$ of $X$ {as in \ref{models}} is called {\it strictly semistable} 
if there is an open covering of $\Xcal$ by open subsets $\Ucal$ 
such that there are \'etale morphisms $\Ucal \to \Spec(\kcirc[x,y]/(xy - \rho_\Ucal))$ for some {$\rho_\Ucal \in K^{\circ\circ}$}. 
Applying the construction in \cite[\S 2.2]{thuillier05:_theor} to the associated formal scheme $\hat{\Xcal}$, 
we get a canonical \emph{skeleton} {$S(\Xcal)\subseteq X^\an$} 
with a proper strong deformation retraction 
$\tau \colon \Xan \to S(\Xcal)$. 
The skeleton $S(\Xcal)$  {carries a canonical structure of} a metrized graph. 
We note that the generic fiber of the formal scheme $\hat{\Ucal}$ intersects $S(\Xcal)$ in an edge  of length $v(\rho_\Ucal)$. 
By using the reduction map, the vertices of $S(\Xcal)$ correspond to the irreducible components of the special fiber $\Xcal_s$
and the open edges of $S(\Xcal)$ correspond to the singular points of $\Xcal_s$.
\end{art} 

\begin{rem}
By definition, a strictly semistable model $\Xcal$ of $X$ is proper over $\kcirc$. 
Using that $X$ is a curve, we will deduce  that $\Xcal$ is projective over $\kcirc$. 
Indeed, the special fiber $\Xcal_s$ is a proper curve over the residue field and hence projective. 
It is easy to construct an effective Cartier divisor $D$ on $\Xcal$ 
whose support intersects any irreducible component of $\Xcal_s$ in a single closed point. 
By \cite[Exercise 7.5.3]{Liu}, the restriction of $D$ to $\Xcal_s$ is ample. 
It follows from \cite[Cor.~9.6.4]{EGA4} that $D$ is ample and hence $\Xcal$ is projective.

Similarly, we can define {strictly semistable formal models}  of $\Xan$. 
Using that $X$ is a smooth projective curve, the algebraization theorem of Grothendieck \cite[Thm.~5.4.5]{EGAIII} and 
its generalizations to the non-noetherian setting  \cite[Cor.~2.13.9]{AbbEGR}, 
\cite[Prop.~I.10.3.2]{FK}  show that formal completion induces an equivalence of categories 
between strictly semistable algebraic models of $X$ and strictly semistable formal models of $\Xan$. 
Here, we need a similar argument as above to construct an effective formal Cartier divisor which restricts to an ample Cartier divisor on the special fiber. 
\end{rem}

\begin{defi} \label{piecewise linear}
A function $f \colon S(\Xcal) \to \R$ is called \emph{piecewise linear} if there is a subdivision of $S(\Xcal)$ 
such that the restriction of $f$ to each edge  of the subdivison is affine. 
We call such an $f$ \emph{integral $\Gamma$-affine} if there is a subdivision such that each edge $e$ has length in $\Gamma$, 
such that $f|_e$ has integer slopes, and such that $f(v) \in \Gamma$ for each vertex $v$ of the subdivision.
\end{defi}

\begin{prop} \label{model functions and skeleton}
Let $\Xcal$ be a strictly semistable model of $X$ and let $f \colon \Xan \to \R$ be a function.
Then the following properties hold:
 \begin{itemize}
\item[(a)] If $f$ is a $\Z$-model function, then $f|_{S(\Xcal)}$ is a  piecewise linear function which is integral $\Gamma$-affine.
\item[(b)] 
The function $f$ is a $\Z$-model function determined on $\Xcal$ 
if and only if {$f=F\circ \tau$ for some function} 
$F \colon S(\Xcal)\to \R $ 
which is affine on each edge of $S(\Xcal)$ with integer slopes and with $f(v) \in \Gamma$ for each vertex $v$ of $S(\Xcal)$. 
\item[(c)] If $G$ is a piecewise linear function on $S(\Xcal)$ which is integral $\Gamma$-affine, then $G \circ \tau$ is a $\Z$-model function.
\end{itemize}
\end{prop}

\begin{proof}
In the $G$-topology on $\Xan$ induced by the strictly $K$-affinoid domains, a $\Z$-model function is given locally by $-\log|\gamma|$ for a rational function $\gamma$ on $X$. 
Hence (a) follows from  \cite[Prop.~5.6]{gubler-rabinoff-werner}. 
Property (b) was proven in \cite[Prop.~B.7]{gubler-hertel} for any dimension. 

To prove (c), we choose a subdivision  of $S(\Xcal)$ as in Definition \ref{piecewise linear} for $G$. 
As in \cite[\S 3]{baker-payne-rabinoff2}, this subdivison is the skeleton of a strictly semistable model $\Xcal'$ dominating $\Xcal$
 and with the same retraction $\tau$. 
Then (c) follows from (b).
Note that the quoted papers have the standing assumption that $K$ is algebraically closed, but it is straightforward to verify that this was not used for the quoted results.
\end{proof}

\begin{art}\label{preparations}
Now we consider a model function $f$ on $\Xan$. 
Using the setting of Proposition \ref{model functions and skeleton} and (b), 
we see that $f = F \circ \tau$ for a piecewise linear function $F$ on $S(\Xcal)$ such 
that $m F$ is integral $\Gamma$-affine for some non-zero $m \in \N$. 
We also assume that $\theta$ is a closed $(1,1)$-form on $\Xan$ which is determined on our given strictly semistable model $\Xcal$. 
\end{art}

We have the following useful characterization  for $f$ to be $\theta$-psh in terms of slopes:

\begin{prop} \label{lemma slope psh}
Under the hypotheses {from \ref{preparations}},  
the model function $f$ 
 is $\theta$-psh 
if and only if $F$ satisfies for all $x \in \Delta  \coloneqq S(\Xcal) $
\begin{align} \label{slope formula}
\sum_{\nu \in T_x(\Delta) } w_x(\nu) \lambda_{x,\nu}(F) + \deg(\theta |_{\Ccal_x}) \geq 0,
\end{align}
where $\nu$ ranges over the {set $T_x(\Delta)$ of} outgoing  tangent  directions at $x$. 
Here, $\lambda_{x,\nu}(F)$ denotes the slope of $F$ at $x$ along $\nu$ and 
we have the weight $w_x(\nu) \coloneqq [\tilde{K}(p_\nu):\tilde{K}]$ for the singularity $p_\nu$ of $\Xcal_s$ 
corresponding to the edge of $S(\Xcal)$ at $x$ in the direction of $\nu$. 
Moreover, if $x$ is a vertex of $S(\Xcal)$, then $\Ccal_x$ denotes  the  corresponding irreducible component $\Ccal_x$ of $\Xcal_s$ and 
if $x$ is not a vertex, then   $\deg(\theta |_{\Ccal_x}) \coloneqq 0$.
\end{prop}

\begin{proof}
If we pass to the completion $\C_K$ of an algebraic closure of $K$, 
there is a strictly semistable model $\Xcal'$ dominating $\Xcal$ such that $f$ is determined on {$\Xcal'$}. 
This is proven in \cite[\S 7]{bosch-luetkebohmert-1985}. 
We note that the property $\theta$-psh holds if and only if the corresponding property holds after base change to $\C_K$.
This is a consequence of the projection formula in algebraic intersection theory. 
Since the degree is invariant under base change, it follows from \cite[Prop.~2.2.21]{thuillier05:_theor} 
that the left hand side of \eqref{slope formula} is invariant under base change as well. 
We conclude that we may assume that $K$ is algebraically closed and that $\Xcal = \Xcal'$, i.e. $f$ is determined on $\Xcal$. 
Then \eqref{slope formula} follows from the slope formula of Katz, Rabinoff, and Zureick-Brown \cite[Thm.~2.6]{katz-rabinoff-zureick}.
\end{proof}

The next result is crucial for the proof of Theorem \ref{thm-continuity-theta-envelope-curves}. It is well-known to the experts, but in our quite general setting we could not find a proof in the literature (a special case was proven in  \cite[B.16]{gubler-hertel}). The result is related to the fact that the retraction from a graph to a subgraph preserves subharmonicity of functions (see for example \cite[Sect.~2.5.1]{jonsson2015} for the case of trees).

\begin{prop} \label{lemma psh retraction curve}
Let $\Xcal$ be a strictly semistable model of $X$ with canonical retraction {$\tau \colon  \Xan \to S(\Xcal)$}.
Let $\theta \in \Zcal^{1,1}(X)$ be determined on $\Xcal$ and  
let $\varphi \colon X^{\an} \to \R$ be an arbitrary $\theta$-psh model function.
Then $\varphi \circ \tau \colon X^{\an} \to \R$ is a  $\theta$-psh model function with $\varphi \leq \varphi \circ \tau$.
\end{prop}
\begin{proof}
It follows from Proposition \ref{model functions and skeleton} that $\varphi \circ \tau$ is a model function. 
To check that $\varphi \circ \tau$ is $\theta$-psh, we may assume $K$ algebraically closed  as 
we have seen in the proof of Proposition \ref{lemma slope psh}. 
Moreover, we have seen that there is a strictly semistable 
model $\Xcal'$ of $X$ dominating $\Xcal$ such that $f$ is determined on $\Xcal'$. 
Then
\begin{align*}
\Delta \coloneqq S(\Xcal) \subset \Delta' \coloneqq S(\Xcal')
\end{align*}
and $\varphi \circ \tau$ is constant along edges of $\Delta'$ which are not contained in $\Delta$.
By Proposition \ref{model functions and skeleton}, 
there is a piecewise linear function $F'$ on $\Delta'$ with $\varphi = F' \circ \tau'$ 
for the canonical retraction $\tau' \colon \Xan \to \Delta'$ such that $m F'$ is integral $\Gamma$-affine for a non-zero $m \in \N$. 
Moreover, the function $F'$ is affine on the edges of $\Delta'$. 
Let $F$ be the restriction of $F'$ to $\Delta$.
The same arguments as in \cite[Prop.~5.7]{BFJ1} 
show that the $\theta$-psh function $\varphi$ is a uniform limit of functions of the form $\frac{1}{m} \log |\mathfrak{a}|$ 
with non-zero $m \in \N$ and with a vertical fractional ideal sheaf $\mathfrak a$ on $\Xcal$. 
By \cite[Thm.~5.2(ii)]{berkovich-1999}, we deduce that
$\varphi \leq \varphi \circ \tau$. 
Using the terminology introduced in Proposition \ref{lemma slope psh}, for all $x \in \Delta$ and $v \in T_x(\Delta')$ we obtain 
$\lambda_{x, \nu}(F') = \lambda_{x,\nu}(F)$ if $v \in T_x(\Delta')$ and 
$\lambda_{x, \nu}(F') \leq 0$ if $\nu  \in T_x(\Delta') \setminus T_x(\Delta)$. 
This implies 
\begin{align*}
0 &\leq  \sum_{\nu \in T_x(\Delta') } w_x(\nu) \lambda_{x,\nu}(F') + \deg(\theta |_{\Ccal_x}) 
\leq \sum_{\nu \in T_x(\Delta) } w_x(\nu) \lambda_{x,\nu}(F) + \deg(\theta |_{\Ccal_x})
\end{align*}
Here this first inequality comes from Proposition \ref{lemma slope psh}, 
since $\varphi$ is $\theta$-psh. 
Applying Proposition \ref{lemma slope psh} again, we conclude that $\varphi \circ \tau = F \circ \tau$ is $\theta$-psh.
\end{proof}

\begin{rem}
Proposition \ref{lemma psh retraction curve} does not hold for higher dimensional varieties.
We refer to the Appendix for a toric counterexample in dimension two by José Burgos and Mart\'{i}n Sombra.
\end{rem}

The following special case of model functions is crucial for the proof of 
Theorem \ref{thm-continuity-theta-envelope-curves}. 
{Especially for model functions}, we can say much more about the envelope.

\begin{prop} \label{prop psh envelope model curve}
Let $\theta$ be a closed $(1,1)$-form with nef de Rham class $\{\theta\}$
on the smooth projective curve $X$  over $K$ and let 
 $f \colon X^{\an} \to \R$ be a model function. 
We assume that $\theta$ and $f$ are determined on the strictly semistable model $\Xcal$ of $X$. 
Let $\tau \colon \Xan \to S(\Xcal)$ be the canonical retraction to the skeleton. 
Then the following properties hold:
\begin{enumerate}
\item There is $F \colon S(\Xcal) \to \R$ which is affine on each edge and with $\pshenvelope_\theta(f)= F \circ \tau$.
\item If $\Gamma \subset \Q$ and if $\theta \in \Zcal^{1,1}(X)_\Q$, then $\pshenvelope_\theta(f)$ is a
$\theta$-psh model function which is determined on $\Xcal$.
\end{enumerate}
\end{prop}

\begin{proof} 
\emph{Step 1}.  
By Proposition \ref{regulari-prop2}(iv),
we have
\begin{align*}
\pshenvelope_\theta(f) = \pshenvelope_{\theta +dd^cf}(0) +f.
\end{align*}
Hence replacing $\theta$ by $\theta + dd^cf$ and $f$ by $0$, 
we can assume that $f=0$ by Proposition \ref{model functions and skeleton}(b).

\emph{Step 2}.
Let $\Delta \coloneqq S(\Xcal)$ denote the skeleton of $\Xcal$.
By Propositions \ref{model functions and skeleton} and \ref{lemma psh retraction curve},  
we get that 
\begin{equation}\label{eq retraction}
\pshenvelope_\theta(0) = \sup_{F \in \Acal} F \circ \tau
\end{equation}
for the set $\Acal$ of non-positive piecewise linear functions $F$ on $\Delta$ such that 
$mF$ is integral $\Gamma$-affine for some $m \in \N_{>0}$ and such that $F\circ \tau$ is $\theta$-psh.  
Note that the piecewise linear functions are not assumed to be affine on the edges of $\Delta$. 
Since $X$ is a smooth projective curve and the de Rham class $\{\theta\}$ is nef, it is clear that $\Acal$ is non-empty.
We introduce the function 
$F_0 \colon \Delta \to \R$ defined by 
\begin{align*}
F_0 := \sup_{F \in \Acal} F.
\end{align*}
By \eqref{eq retraction}, we get that $\pshenvelope_\theta(0) = F_0 \circ \tau$.
Hence we can reduce (i) to prove that $F_0$ is affine on each edge of $\Delta$.

\emph{Step 3}.
 For $F \in \Acal$, let $L(F) \colon \Delta \to \R$ be the function which is affine on the edges of $\Delta$ and 
which agrees with $F$ on the set $V$ of vertices of $\Delta$. 
As $F \leq 0$, we deduce immediately $L(F) \leq 0$. 
Since $F \circ \tau$ is $\theta$-psh, Proposition \ref{lemma slope psh} shows that $F$ is convex on each edge of $\Delta$ and hence $F \leq L(F)$.

By passing from $F$ to $L(F)$, the slopes do not decrease in the vertices and using that $F \circ \tau$ is $\theta$-psh, 
it follows from Proposition \ref{lemma slope psh} that $L(F) \circ \tau$ is $\theta$-psh as well.

The slopes of the function $L(F)$ might be non-rational. 
However, we can approximate the slopes of $L(F)$ in a rational way at any vertex 
and thus for any $\varepsilon > 0$ we find a piecewise 
linear function $L_{\varepsilon}(F)$ on $\Delta$ such that
\begin{enumerate}
\item
 $L_\varepsilon(F)$ agrees with $F$ on $V$.
\item
$L_\varepsilon(F)$ has rational slopes.
\item
$L_\varepsilon(F)$ is convex on the edges of $\Delta$.
\item
$L_\varepsilon(F) \geq F$.
\item
$\sup \vert L_\varepsilon(F) - L(F) \vert < \varepsilon$
\end{enumerate}

We claim that $L_\varepsilon(F) \in \Acal$. 
It follows from (ii) that $mL_\varepsilon(F)$ is integral $\Gamma$-affine for some $m \in \N_{>0}$. 
Since $F$ and $L(F)$ agree on $V$, it is clear from (i) and (iii) that $L_\varepsilon(F) \leq L(F) \leq 0$. 
To show that $L_\varepsilon(F) \circ \tau$ is $\theta$-psh, we use the slope criterion from Proposition \ref{lemma slope psh}. 
Note that \eqref{slope formula} is fulfilled in the interior of each edge of $\Delta$ by (iii).  
In a vertex of  $\Delta$, the inequality \eqref{slope formula} is satisfied by using the corresponding inequality for $F$, (i) and (iv). 
This proves $L_\varepsilon(F) \in \Acal$.
As a consequence we find
\begin{align} \label{linearized maximumsupII}
F_0 = \sup_{F \in \Acal} F = \sup_{F \in \Acal} L_\varepsilon(F) = \sup_{F \in \Acal} L(F).
\end{align}

\emph{Step 4}. 
We claim $F_0 = L(F_0)$.
First note that since the $\max$ of convex functions in convex, $F_0$ is convex on each edge of $\Delta$, 
thus $F_0 \leq L(F_0)$.

We pick an $\varepsilon > 0$. For any $v \in V$, there is $f_v \in \Acal$ with  $f_v(v) >  L(F_0)(v) - \varepsilon$. 
It follows from \cite[Prop.~3.12]{gubler-martin} that the maximum of two $\theta$-psh model functions is again a $\theta$-psh model function. 
Using \ref{max psh}, we conclude that $\Acal$ is closed under the operation $\max$. 
Thus $L(\max \{f_v \mid v \in V\}) \in \Acal$ is in $\varepsilon$-distance to $L(F_0)$ at every vertex of $\Delta$ and 
hence at every point of $\Delta$. 
As $\varepsilon >0$ can be chosen arbitrarily small,  
\eqref{linearized maximumsupII} yields $L(F_0) \leq F_0$ and hence we get Step 4. 
Note that Step 4 proves (i).

\emph{Step 5}. 
{In the case $\Gamma \subset \Q$, we have  $L(F) \in \Acal$ for $F \in \Acal$. 
Indeed, we note that in this special case the edges have rational lengths and $F$ takes rational values at $V$. 
We deduce from Proposition \ref{model functions and skeleton} that 
$L(F)\circ \tau$ is a model function.}
Let $\Bcal$ be the set of 
 $F \in \Acal$ such that $F$ is affine on every edge of $\Delta$. 
For $F \in \Acal$ we thus have $L(F) \in \Bcal$ 
which shows via \eqref{linearized maximumsupII} that we might restrict the $\sup$ to $\Bcal$ in the definition of $F_0$.
Recall that $V$ is the set of vertices of $\Delta$.
Since the edge lengths of $\Delta$ are rational,  
the map $\Psi \colon \Bcal \to \R^V$ defined by $\Psi(F) = (F(v))_{v\in V}$ 
identifies $\Bcal$ with the rational points of a rational polyhedron in $\R^V$ 
defined by the linear inequalities of Proposition \ref{lemma slope psh}.
Note that by affineness on the edges, we need to check the slope inequalities only at the vertices. 
If a rational linear form $\varphi$ is bounded from above on a rational polyhedron $P$, 
then $\varphi|_P$ achieves its maximum in a rational point. 
Hence there exists $G \in \Bcal$ 
such that 
\begin{align} \label{eq max}
\sum_{v\in V} G(v) = \max_{F\in \Bcal} \left( \sum_{v\in V} F(v)  \right).
\end{align}
We claim that $G= F_0 $. 
Considering $F\in \Bcal$, we get $\max(G,F) \in \Acal$ by Step 4 and hence  
$H' \coloneqq  L(\max(G,F)) \in \Bcal$. 
Hence $H'\geq G$, $H' \geq F$ and $H'\in \Bcal$.
But by \eqref{eq max} we deduce that for $v\in V$ we have $H'(v) = G(v)$.
Since functions in $\Bcal$ are determined by their values on $V$, we have $H'=G$.
Hence $G\geq F$, whence $G=F_0$. 
It follows that $P_\theta(0)= F_0 \circ \tau = G \circ \tau$ is a $\theta$-psh function proving (ii).
\end{proof}

\begin{proof}[Proof of Theorem \ref{thm-continuity-theta-envelope-curves}]
By Proposition \ref{regulari-prop2} (viii) 
it is enough to prove the continuity of $\pshenvelope_\theta(u)$. 
By the semistable reduction theorem \cite[\S 7]{bosch-luetkebohmert-1985}, 
there is a finite field extension $K'/K$ such that $X' \coloneqq X \otimes_K  K'$ has a strictly semistable model $\Xcal'$ with 
$\theta' \coloneqq q^*\theta$ determined on $\Xcal'$, where $q \colon X' \to X$ is the canonical map. 
It follows from Proposition \ref{prop psh envelope model curve} that $\pshenvelope_{\theta'}(u \circ q)$ is continuous. 
We know from Lemma \ref{lemma A4 BFJ2} that 
\begin{align*}
\pshenvelope_{\theta'}(u \circ q) = q^*(\pshenvelope_\theta(u)).
\end{align*}
By \cite[Prop.~1.3.5]{berkovich-book}, the topological space of $\Xan$ is  the quotient of $(X')^{\rm an}$ by the automorphism group of $K'/K$. 
We conclude that $\pshenvelope_\theta(u)$ is continuous.
\end{proof}

In the following, we consider an ample line bundle $L$ on the projective smooth curve $X$ over $K$.
Recall that we have defined the semipositive envelope $\pshenvelope(\| \ \|)$ 
of a continuous metric $\metr$ on $\Lan$ in \eqref{envelope-metric}.
\begin{cor} \label{cor rationality}
Assume that $\Gamma \subset \Q$. 
Let $\| \ \|$ be a model metric on $L$. 
Then $\pshenvelope(\| \ \|)$ is a semipositive model metric on $L$. 
\end{cor}

\begin{proof}
By definition, we have
$\pshenvelope(\| \ \|) = \| \ \| e^{-\pshenvelope_\theta(0)}$ for $\theta \coloneqq c_1(L,\| \ \|)$, 
hence the claim follows from Proposition \ref{prop psh envelope model curve}.
\end{proof}

From now on, we assume that $K$ is discretely valued. 
The goal is to prove some rationality results for the non-archimedean volumes  on the line bundle $L$ of the smooth projective curve $X$ over $K$.  
Non-archimedean volumes $\vol(L, \metr_1, \metr_2)$ with respect to continuous metrics $\metr_1, \metr_2$ on $\Lan$ 
are analogues of volumes $\vol(L)$ in algebraic geometry. 
We refer to \cite[Def.~4.1.2]{BGJKM} for the precise definition. 
By the Riemann--Roch theorem,  $\vol(L) \in \Q$  in the special case of curves. 
We will show a similar result about non-archimedean volumes.

\begin{cor} \label{rational volume}
Let $K$ be a field endowed with a complete discrete valuation with value group $\Gamma \subset \Q$. 
Let $\| \ \|_1$ and $\| \ \|_2$ be two model metrics on the line bundle $L$ of the smooth projective curve $X$ over $K$.
Then $\vol(L,\| \ \|_1,\| \ \|_2) \in \Q$.
\end{cor}

\begin{proof}
If $\deg(L)\leq 0$, then it is clear from the definition that $\vol(L,\| \ \|_1,\| \ \|_2)=0$. 
So we may assume that $L$ is ample.
We need the energy $E(L, \metr_1, \metr_2)$ with respect to continuous semipositive metrics 
$\metr_1, \metr_2$ on $\Lan$ introduced in \cite[Def.~2.4.4]{BGJKM}. 
By Corollary \ref{cor rationality}, the envelopes $\pshenvelope(\metr_1)$ and $\pshenvelope(\metr_2)$ are semipositive model metrics on $\Lan$. 
In particular, they are continuous and hence it follows from  \cite[Cor.~6.2.2]{BGJKM} that
\begin{align*}
\vol(L, \metr_1, \metr_2) = E(L,\pshenvelope(\metr_1),\pshenvelope(\metr_2)).
\end{align*}
In the case of semipositive model metrics associated to line bundles $\Lcal_1, \Lcal_2$ on a $\kcirc$-model $\Xcal$, 
our assumption $\Gamma \subset \Q$ yields that the energy is defined as a $\Q$-linear combination 
of intersection numbers of the line bundles $\Lcal_1, \Lcal_2$ on $\Xcal$ proving the claim.
\end{proof}

\begin{rem} \label{rem rationality}
When $\dim(X) \geq 3$, there are varieties with line bundles $L$ such that $\vol(L)$ is irrational
(see \cite[Example 2.2]{ELMN05} or \cite[Example 2.3.8]{Laz1}). 
Hence with our definition and normalization of non-archimedean volumes, we get for a model metric $\| \ \|$ that 
$\vol(L,\| \ \|, e^{\lambda}\| \ \|) = \vol(L)\lambda$ which produces irrational non-archimedean volumes.
The following natural questions remain open:
\begin{enumerate}
\item 
What happens if $X$ is a variety of 
dimension two? Are non-archimedean volumes rational?
Note that by Zariski decomposition, $\vol(L)$ is rational then  
(see for instance \cite[Cor.~2.3.22]{Laz1}).
\item If we normalize our non-archimedean volumes by $\vol(L)$, can we find an example of some model metrics $\| \ \|$ and $\| \ \|'$ with irrational 
$\vol(L,\| \ \|, \| \ \|')$?
The idea is to avoid the trivial example above.
\end{enumerate}
\end{rem}

\section{Asymptotic test ideals}\label{app-test-ideals}\label{section4}

We recall definitions and some basic properties
from the theory of generalized and asymptotic
test ideals developed in \cite[Sect.~2]{blickle-mustata-smith08}
and \cite[Sect.~3]{mustata13}.
{We refer to \cite{SchTuc} for a more comprehensive overview of the 
theory of test ideals.}
Let $X$ be a smooth variety over a perfect field $k$
of characteristic $p>0$.
Let $F\colon X\to X$ denote the Frobenius morphism which is
induced by the $p$-th power ring morphism on affine subsets.
Write
\[
\omega_{X/k}=\det\Omega_{X/{k}}=\KO_X(K_{X/k})
\] 
for some canonical divisor $K_{X/k}$ on $X$.

Let $\Da$ be an ideal in $\KO_X$ and $e\in\N_{>0}$.
There is a unique ideal $\Da^{[p^e]}$ in $\KO_X$ such that 
for every open affine $U$ in $X$ the ideal $\Da^{[p^e]}(U)$
in $\KO_X(U)$ is generated by $\{u^{p^e}\,|\,u\in \Da(U)\}$.
We have \cite[bottom of p.~44]{blickle-mustata-smith08}
\begin{equation}\label{simple-test-property-1}
\Da(U)=\{a\in \KO_X(U)\,|\,a^{p^e}\in \Da^{[p^e]}(U)\}
\end{equation}
We recall the following facts from \cite[p.~540]{mustata13}: 
There is a canonical trace map
${\rm Tr}\colon F_*(\omega_{X/k})\to \omega_{X/k}$
whose construction can be based on the Cartier isomorphism 
\cite[Thm.~(7.2) and Eq.~(7.2.3)]{katz70}. 
Mustaţă gives an explicit description of the trace map 
\cite[top of p.~540]{mustata13}.
Given $e\in\N_{>0}$ there is an
iterated trace map
${\rm Tr}^e \colon F^e_*(\omega_{X/k}) \to \omega_{X/k}$.
For an ideal $\mathfrak a$ in $\KO_X$
there exists a unique ideal  
$\mathfrak a^{[1/p^e]}$ in $\KO_X$ with 
\begin{equation}\label{mustata-trace-operation-eq1}  
{\rm Tr}^e(F^e_*(\mathfrak a \cdot \omega_{X/k})) 
= \mathfrak a^{[1/p^e]} \cdot \omega_{X/k}.
\end{equation}
This definition of $\mathfrak a^{[1/p^e]}$
is compatible with \cite[Def.~2.2]{blickle-mustata-smith08}.
Hence we have 
\begin{equation}\label{simple-test-property-2}
\bigl(\Da^{[p^{e}]}\bigr)^{[1/p^e]}=
\Da\subseteq \bigl(\Da^{[1/p^{e}]}\bigr)^{[p^e]}
\end{equation}
by 
\cite[Lemma 2.4(iv)]{blickle-mustata-smith08}.

\begin{defi}\label{def-test-exp}\cite[Def.~2.9]{blickle-mustata-smith08}
Given an ideal $\Da$ in $\KO_X$ and $\lambda\in\R_{\geq 0}$
one defines the \emph{test ideal of $\Da$ of exponent $\lambda$} to be
\[
\tau(\Da^\lambda):=\bigcup_{e\in \N_{>0}}
\Bigl(\Da^{\lceil \lambda
p^e\rceil}\Bigr)^{[1/p^e]}
\]
where given $r\in\R$ we write $\lceil r\rceil$ 
for the smallest integer $\geq r$.
\end{defi}

\begin{rem}\label{rem-def-tau}
\begin{enumerate}[itemindent = 10 mm, leftmargin =  0 mm]
\item
Observe that we have
$\tau(\Da^\lambda)=(\Da^{\lceil \lambda p^e\rceil})^{[1/p^e]}$ for large
$e\in\N$ as $X$ is noetherian.
The equality
\begin{equation}
\tau((\Da^m)^\lambda)=\tau(\Da^{\lambda m})
\end{equation}
for $m\in\N$
shows that the notation in Definition \ref{def-test-exp}
is compatible with taking powers of ideals
\cite[Cor.~2.15]{blickle-mustata-smith08}.
We have
$\tau(\Da^\lambda)\subseteq \tau(\Db^\lambda)$ for ideals
$\Da\subseteq \Db$ in $\KO_X$ \cite[Prop.~2.11(i)]{blickle-mustata-smith08}.

\item
Choose $e$ such that 
$\tau(\Da)=(\Da^{[p^e]})^{[1/p^e]}$.
For any ideal $\Db$ in $\KO_X$ such that 
$\Da^{[p^e]}\subseteq \Db^{[p^e]}$ 
we get $\Da\subseteq \Db$ from \eqref{simple-test-property-1}.
Hence \eqref{simple-test-property-2} implies
\begin{equation}\label{fundamental-inclusion-1}
\Da\subseteq \tau(\Da).
\end{equation}

\end{enumerate}
\end{rem}

Let $\Da_\bullet$ be \emph{graded sequence of ideals in $\KO_X$},
i.e.~a family $(\Da_m)_{m\in\N_{>0}}$ of ideals in $\KO_X$  such that
$\Da_m\cdot\Da_n\subseteq \Da_{m+n}$ for all $m,n\in \N_{>0}$
and $\Da_m\neq (0)$ for some $m>0$.

\begin{defi}\cite[p.~541]{mustata13}
Choose $\lambda\in \R_{\geq 0}$.
Define the \emph{asymptotic test ideal of exponent $\lambda$} as
\[
\tau(\Da_\bullet^\lambda):=
\bigcup_{m\in\N}\tau(\Da_m^{\lambda/m}).
\]
\end{defi}

\begin{rem}\label{remark-asymptotic}
\begin{enumerate} [itemindent = 10 mm, leftmargin =  0 mm]
\item
We have
$\tau(\Da^\lambda_\bullet)=\tau(\Da_m^{\lambda/m})$ for suitable $m\in\N$
which are divisible enough \cite[p.~541]{mustata13}.

\item
{ 
For all $m\in\N$ we have \cite[p.~541, l.~4]{mustata13}}
\begin{equation}\label{fundamental-inclusion-2}
\tau(\Da_m)\subseteq \tau (\Da^m_\bullet).
\end{equation} 
\item
For all $m\in\N$ we have the 
\emph{Subadditivity Property} \cite[Prop.~3.1(ii)]{mustata13}
\begin{equation}\label{subadditivity-2}
\tau(\Da_\bullet^{m\lambda})
\subseteq \tau(\Da_\bullet^\lambda)^m.
\end{equation}
\end{enumerate}
\end{rem}

\begin{defi}
Let $D$ be a divisor on $X$ with
$h^0(X,\KO_X(mD))\neq 0$ for some $m>0$.
Define the \emph{asymptotic test ideal of exponent $\lambda\in\R_{\geq 0}$
associated with $X$ and $D$ as}
\[
\tau(\lambda\cdot \|D\|):=\tau(\Da_\bullet^\lambda)
\]
where $\Da_\bullet$ denotes the graded sequence of
base ideals for $D$, i.e.~$\Da_m$ is the image of the
natural map
\[
H^0(X,\KO(mD))\otimes_k \KO_X(-mD)\to \KO_X.
\]
If $D$ is a $\Q$-divisor such that 
$h^0(X,\KO_X(mD))\neq 0$ for some positive integer 
$m$ such that $mD$ is a usual divisor then
we put $\tau(\lambda\cdot\|D\|):=\tau(\lambda/r\cdot\|rD\|)$
for some $r\in \N$ such that $rD$ has integral coefficients. 
\end{defi}

We finish with a slight generalization of Mustaţă's uniform 
generation property
\cite[Thm.~4.1]{mustata13}.
Observe that in \emph{loc.~cit.}~it is required that
the variety $X$ is projective over the ground field $k$.

\begin{theo} \label{uniform generation}
Let $R$ be a $k$-algebra of finite type
over a perfect field $k$ of characteristic $p>0$.
Let $X$ be an integral scheme of dimension $n$ 
which is projective over the spectrum of $R$ 
and smooth over $k$. 
Let $D$, $E$, and $H$ be divisors on $X$ and 
$\lambda\in\Q_{\geq 0}$
such that
\begin{enumerate}
\item
$\KO_X(H)$ is an ample, globally generated line bundle,
\item
$h^0(X,\KO_X(mD))>0$ for some $m>0$, and
\item
the $\Q$-divisor $E-\lambda D$ is nef.
\end{enumerate}
Then the sheaf
$\Ocal_X(K_{X/k}+E+dH) 
\otimes_{\Ocal_X} \tau(\lambda \cdot \|D\|)$ 
is globally generated for all $d \geq n+1$.
\end{theo}

\begin{proof} 
We literally follow Mustaţă's proof with two modifications. 
The proof requires Mumford's theorem on Castelnuovo-Mumford 
regularity for the projective scheme $X$ over $R$  which holds 
also in this more general setting  \cite[20.4.13]{brodmann-sharp}.
Furthermore 
we replace the use of Fujita's vanishing theorem
to the sheaves 
$\Fcal_j := \Ocal_\Xcal(K_{X/k} + T_j)$, $j=1, \dots, r$
and the ample divisor $(d-i)H$
by an application of Keeler's {generalization} 
\cite[Thm.~1.5]{keeler03}.
\end{proof}

\section{{Descent for model functions}}\label{new-section-5}

Let $K$ denote a complete discretely valued field
with valuation ring $K^\circ$.
Let $R$ be a discrete valuation 
subring of $K^\circ$ whose completion is $K^\circ$.
Then $K$ is the completion of the field of fractions
$F$ of $R$.
In this section we show that all model functions on
analytifications of varieties over $K$ are already defined over $R$.

An $R$-model of a 
projective variety over $F$ is defined completely analogously 
to the complete case treated in \ref{models}.

\begin{defi}
\label{definition model function over model}
Let $X$ be a projective variety over $K$.
We say that a \emph{model function} $\varphi \colon X^\an \to \R$ 
\emph{is defined over $R$} if there exists 
a projective variety $Y$ over $F$, with an isomorphism 
$Y\otimes_F K \simeq X$, 
an $R$-model $\KY$ of $Y$, a vertical divisor $D_0$ on $\KY$ such that 
$\varphi = \frac{1}{m} \varphi_D$
 where $m \in \N_{> 0}$ and $D$ is the vertical 
 divisor on $\KY\otimes_R K^\circ$ obtained by pullback from $D_0$.
Likewise we define the notion of a 
\emph{vertical ideal sheaf defined over $R$}.
\end{defi}

Here is the announced descent result.

\begin{prop} \label{models are geometric}
Let $Y$ be a projective 
variety over $F$ and let $X \coloneqq Y \otimes_F K$. 
\begin{itemize}
 \item[(a)] 
 Any $K^\circ$-model of $X$ is dominated by the base change of a projective $R$-model of $Y$ to $K^\circ$.
 \item[(b)] 
 If  a projective $K^\circ$-model $\KX$ of $X$ dominates 
 $\KY \otimes_R K^\circ$ for a projective $R$-model $\KY$ 
 of $Y$, then $\KX \simeq \KY' \otimes_R K^\circ$ for a 
 projective $R$-model $\KY'$ of $Y$ dominating  $\KY$.
 \item[(c)] 
 Every model function on $\Xan$ is defined over $R$.
\end{itemize}
\end{prop}

\begin{proof} 
To prove (a), we pick any projective $R$-model $\KY$ of $Y$. 
By \cite[Lemma 2.2]{luetkebohmert1993}, there is a blowing up $\pi \colon\KX' \to \KY \otimes_R K^\circ$ 
such that $\KX'$ dominates $\KX$. Since $\pi$ is a projective 
morphism, $\KX'$ is a projective $K^\circ$-model dominating $\KX$. 
Hence (a) follows from (b). 

To prove (b), we note that the morphism $\KX \to \KY \otimes_R K^\circ$ 
is a blowing up morphism along a vertical closed subscheme $Z$ of 
$\KY \otimes_R K^\circ$ (see \cite[Thm.~8.1.24]{Liu}). 
Since the ideal sheaf of $Z$ contains a power of the 
uniformizer of $R$, we may define it over $R$ and hence 
the same is true for the blowing up morphism and for $\KX$ 
proving (b). 

To prove (c),  we may assume that the model function is associated to a vertical 
Cartier divisor $D$. Replacing $D$ by $D + \div(\lambda)$ for a suitable 
non-zero $\lambda \in R$ and using (a) and (b), we may assume that $D$ is an  effective Cartier divisor 
on a projective 
$R$-model $\KY$ of $Y$.  
As in (b), we 
see that the ideal sheaf of $D$ is defined by the ideal sheaf of a Cartier divisor $D_0$ defined 
over $R$ proving (c).
\end{proof}

\section{Resolution of singularities}\label{new-section-6}

For our applications, we need that regular projective models 
are cofinal in the categroy of models
which makes it necessary to assume 
resolution of singularities in a certain dimension.

\begin{defi} \label{resolution of singularities}
Let $k$ be a field. We say that \emph{resolution of singularities
holds over $k$ in dimension $n$} if for every quasi-projective
variety $Y$ over $k$ of dimension $n$
there exists a regular variety
$\tilde Y$ over $k$ and a projective morphism $\tilde Y\to Y$ which is an 
isomorphism over the regular locus of $Y$.
\end{defi}

To transfer results from \cite{BFJ1, BFJ2} to our context, it is 
essential to show that projective models 
are dominated by SNC-models. 
In order to this we are going to use the following assumption.

\begin{defi} \label{embedded resolution of singularities}
We say that {\it embedded  resolution of singularities  in 
dimension $m$} holds over a field $k$ if for every 
quasi-projective regular variety $Y$ over $k$ 
of dimension $m$ 
and every proper closed subset $Z$ of $Y$, there is a projective 
morphism $\pi:Y' \to Y$ of quasi-projective regular varieties 
over $k$ such that the set $\pi^{-1}(Z)$ is the support 
of a normal crossing divisor and such that $\pi$ is an 
isomorphism over $Y \setminus Z$. 
\end{defi}

Hironaka has shown that resolution of
singularities and embedded resolution of singularities
holds over a field of characteristic zero
in any dimension.
Resolution of singularities holds over 
arbitrary fields in dimension one
(Dedekind, M.~Noether, Riemann)
and in dimension two (Abhyankar, Lipman). 
Cossart and Piltant have proven 
that resolution of singularities and embedded resolution of singularities
hold in dimension three over perfect fields.

\begin{theo}[Cossart-Piltant]\label{thm-cp}
Resolution of singularities and embedded resolution of 
singularities hold in dimension three over any perfect 
field.
\end{theo}

\proof
This is shown in 
\cite[Thm.~on p.~1839]{cossart-piltant2009} 
and \cite[Prop.~4.1]{cossart-piltant2008}.
\qed

\section{Uniform convergence to the envelope of the zero function}
\label{section uniform convergence}\label{section5}

Let $K$ be a complete discretely valued field
of positive characteristic $p>0$.
Let $X$ be a smooth projective variety over $K$,
$L$ an ample line bundle on $X$, and
$(\Xcal,\Lcal)$ a model of $(X,L)$ over $K^\circ$.
For $m\in \N_{>0}$ let $\mathfrak{a}_m$
denote the $m$-th base ideal of $\Lcal$ as in 
\eqref{def-base-ideal-rel}.
For a scheme $B$, we recall our convention 
$B^{(1)}=\{p\in B\,|\,\dim \Ocal_{B,p}=1\}$.

\begin{assumption} \label{Geometric assumption}
There exist  
a normal affine variety $B$ over a perfect field $k$,
a codimension one point $b\in B^{(1)}$, 
a projective regular integral scheme $\Xcal_B$ over $B$, 
and line bundles $\Lcal_B$ and $\Acal_B$ over $\Xcal_B$ such 
that there exist
\begin{enumerate}
\item
a flat morphism 
$h\colon \Spec K^\circ\to \Spec \Ocal_{B,b}\to B$, 
\item 
an isomorphism
$\Xcal_B \otimes_B \Spec K^\circ \stackrel{\sim}{\to} \Xcal$, 
\item
an isomorphism $h^*\Lcal_B \stackrel{\sim}{\to} \Lcal$
over the isomorphism in {\rm (ii)}, 
\item
and an isomorphism 
$\Acal_B|_{\Xcal_{B,\eta}}\stackrel{\sim}{\to} 
\Lcal_B|_{\Xcal_{B,\eta}}$
where $\eta$ is the generic point of $B$ and
the line bundle $\Acal_B$ on $\Xcal_B$ is ample. 
\end{enumerate}
Usually we read all the isomorphisms above as identifications. 
\end{assumption}

Note that all relevant information in 
Assumption \ref{Geometric assumption} is over the 
discrete valuation ring $\Ocal_{B,b}$. 
The next remark makes this statement precise and 
gives an equivalent local way to formulate 
this assumption.

\begin{rem} \label{equivalent geometric assumption}
Suppose that $\Xcal$ and $\Lcal$ are defined over a 
subring $R$ of $\kcirc$ by a line bundle $\Lcal_R$ on 
a projective regular integral scheme $\Xcal_R$ over $R$. 
We assume furthermore that $R$ is a discrete valuation ring 
which is {\it defined geometrically by a $d$-dimensional normal 
variety $B$ over a field $k$}, i.e. there exist $b \in B^{(1)}$ 
and an isomorphism $h\colon R \stackrel{\sim}{\to} \Ocal_{B,b}$. 
We read the isomorphism $h$ as an identification. 
Then Assumption \ref{Geometric assumption} is equivalent to the 
existence of data $(R,k,B,b,h,\Xcal_R,\Lcal_R)$ as above
assuming furthermore  that the field $k$ is perfect and the restriction 
of $\Lcal_R$ to the generic fiber $\Xcal_{R,\eta}$ over $R$ extends 
to an ample line bundle $\Acal_R$  on $\Xcal_R$. 

One direction of the equivalence is clear by base change 
from $B$ to $\Spec \Ocal_{B,b}$. 
On the other hand, 
replacing $B$ by an open affine neighbourhood of $b$,
it is clear by \cite[Cor.~9.6.4]{EGA4} that  
$\Acal_R$ extends to 
an ample line bundle $\Acal_B$ on a projective integral 
scheme $\Xcal_B$ over $B$ 
and that $\Lcal_R$ extends to a line bundle $\Lcal_B$ on $\Xcal_B$.
Since the regular locus of $\Xcal_B$ is open 
\cite[Cor.~12.52]{goertz-wedhorn-1} and since the 
fiber of $\Xcal_B$ over $b$ is contained in the regular locus, 
we may assume that $\Xcal_B$ is also regular by shrinking $B$ 
again.
\end{rem}

\begin{theo} \label{relative Theorem 8.5}
Let $\theta$ be defined by the line bundle $\Lcal$. 
If the pair $(\Xcal,\Lcal)$ satisfies Assumption \ref{Geometric assumption},  
then $({m}^{-1}\log|\mathfrak{a}_m|)_{m\in\N_{>0}}$ is a 
sequence of $\theta$-psh 
model functions which converges uniformly on $X^\an$ to $\pshenvelope_\theta(0)$. 
\end{theo}

If the field $K$ has equicharacteristic zero, 
this result was proven by 
Boucksom, Favre, and Jonsson \cite[Thm.~8.5]{BFJ1} without 
Assumption \ref{Geometric assumption}. 
We will follow their 
strategy of proof replacing 
the use of multiplier ideals by the use of test ideals. 
The required results about test ideals are gathered in  
Section \ref{app-test-ideals}.

\begin{proof}
We start with the observation that we have
\begin{equation} \label{one global section}
\Gamma(\Xcal_B, \Lcal_B^{\otimes m}) \neq 0
\end{equation}
for some $m>0$.
In fact we have
\[
\Gamma(\Xcal_B, \Lcal_B^{\otimes m})\otimes_RK 
\stackrel{\sim}{\longrightarrow}
\Gamma(X, L^{\otimes m}) \neq 0
\]
by flat base change and the ampleness of $L$
for some $m>0$.

We have a cartesian diagram
\[
\xymatrix{
\Xcal \ar[d] \ar[r]^g &  \ar[d]\ar[rd] \Xcal_B \\
\Spec K^\circ \ar[r]^h & B \ar[r] & \Spec k.
}
\]
We observe that $\Xcal_B$ is a smooth variety over the perfect field 
$k$ and write 
\[
\Da_{B,m}=\mbox{Im} \bigl(H^0(\KX_B,\KL_B^{\otimes m})\otimes_k\KL_B^{\otimes -m}
\to \KO_{\KX_B}\bigr)
\]
for the $m$-th base ideal of $\Lcal_B$.
Consider the ideal $g^{-1}(\Da_{B,m})\cdot \KO_\KX$ 
in $\KO_{\KX}$ generated by $g^{-1}(\Da_{B,m})$.
We have $g^{-1}(\Da_{B,m})\cdot \KO_\KX=g^*\Da_{B,m}$
as $g$ is flat.  
Sections of $\Da_m$ are locally of the form
$s\cdot t^{-1}$ where $s\in \Gamma(\KX,\KL^{\otimes m})$
is a global section and $t$ is a local section of 
$\KL^{\otimes m}$.
Flat base change \cite[Prop.~III.9.3]{Hart} gives
\[
H^0(\Xcal, \Lcal^{\otimes m}) = H^0(\Xcal_B, \Lcal_B^{\otimes m}) 
\otimes_R K^\circ.
\] 
Hence the formation of base ideals is compatible with
base change, i.e.~we have
\begin{equation}\label{base-change-base-ideals}
\Da_m=g^{-1}(\Da_{B,m})\cdot \KO_\KX =g^*\Da_{B,m}
\end{equation}
for all $m\in\N_{> 0}$.

The family $\Da_{B,\bullet}=(\Da_{B,m})_{m>0}$ defines a 
graded sequence of ideals 
in the sense of 
Section 
 \ref{app-test-ideals}.
Let $\Db_{B,m}:=\tau(\Da_{B,\bullet}^m)$ denote the 
associated asymptotic test ideal of exponent $m$.
Motivated by \eqref{base-change-base-ideals} we define
\[
\Db_m:=g^{-1}\Db_{B,m}\cdot\KO_{\KX} = g^*\Db_{B,m}
\]
as the ideal in ${\KO_\KX}$ generated by $\Db_{B,m}.$
These ideals have the following properties:
\begin{enumerate}
\item[(a)]
We have $\fa_m \subset \fb_m$ for all 
$m\in\N_{>0}$.
\item[(b)]
We have $\fb_{ml} \subset \fb_m^l$ for all 
$l,m\in\N_{>0}$.
\item[(c)]
There {is $m_0 \geq 0$} such that 
$\Acal^{\otimes m_0}  \otimes \Lcal^{\otimes m} \otimes \mathfrak{b}_m$ 
is globally generated for all {$m>0$.}
\end{enumerate}
Properties (a) and (b) follow from the corresponding
properties of 
$\fa_{B,m}$ and $\fb_{B,m}$ 
 mentioned 
in \eqref{fundamental-inclusion-1},
\eqref{fundamental-inclusion-2}, and \eqref{subadditivity-2} 
if we observe \eqref{base-change-base-ideals}.

Property (c) is a consequence of the generalization of 
Mustaţă's uniform generation property given in 
Theorem \ref{uniform generation}.
Write $\KL_B=\KO(D)$ for some divisor $D$ on $\KX_B$
and choose a divisor $H$ on $\KX_B$ such that
$\KO(H)$ is ample and globally generated.
Fix $d>\dim \KX_B$ and a canonical divisor $K_{\KX_B/k}$
on the smooth $k$-variety $\KX_B$.
As $\Acal_B$ is ample we find some $m_0\in\N$ such 
that $\Acal_B^{\otimes m_0} \otimes \Ocal(-K_{X/k} - dH)$ 
is globally generated.
Given $m\in \N_{>0}$ we put $E:=mD$.
Since $\Lcal_B$ satisfies \eqref{one global section}, 
for any $m\in \N_{>0}$ we may use
 $E:=mD$ and $\lambda := m$ in 
Theorem \ref{uniform generation} to see that the sheaf
\begin{align*}
\Ocal (K_{\KX_B/k} +  dH) \otimes \Lcal_B^{\otimes m}
\otimes\mathfrak{b}_{B,m} 
\end{align*} 
is globally generated.
As a consequence, our choice of $m_0$ implies that
$\Acal_B^{\otimes m_0}  \otimes \Lcal_B^{\otimes m} 
\otimes \mathfrak{b}_{B,m}$ is globally
generated.  
Base change to $K^\circ$ 
proves (c).

{Now we follow the proof of \cite[Thm.~8.5]{BFJ1}. 
Step 1 of \emph{loc.~cit.}~holds not only on quasi-monomial 
points of $\Xan$, but pointwise on the whole $\Xan$ using 
Proposition \ref{regulari-prop5} and our different 
definition of $P_\theta(0)$. 
Then Step 2 of \emph{loc.~cit.}~works in our setting  
using properties (a), (b), and (c) above. 
The only difference is that all inequalities hold 
immediately on $\Xan$ and not only on the quasi-monomial 
points of $\Xan$.}
\end{proof}

\begin{cor} \label{zero envelope}
Let $X$ be a smooth $n$-dimensional projective variety 
over $K$ with a closed $(1,1)$-form $\theta$. 
Let $\Lcal$ be a line bundle on a  $\kcirc$-model $\Xcal$ 
of $X$ defining $\theta$ and with $L=\Lcal|_X$ ample. 
We assume that $(\Xcal,\Lcal)$ is the base change of 
$(\Xcal_R, \Lcal_R)$ for a line bundle $\Lcal_R$ of a 
projective integral scheme $\Xcal_R$ over a subring 
$R$ of $\kcirc$ and that $R$ is a discrete valuation 
ring defined geometrically by a $d$-dimensional normal 
variety $B$ over a perfect field $k$
(as in Remark \ref{equivalent geometric assumption}). 
If resolution of singularities holds over $k$ in 
dimension $d+n$, then $\pshenvelope_\theta(0)$
is a uniform limit of $\theta$-psh 
model functions and hence $\pshenvelope_\theta(0)$ 
is continuous on $\Xan$.
\end{cor}

\begin{proof}
It follows from our assumptions that $X$ is base change 
of the generic fiber $\Xcal_{R,\eta}$ 
of $\Xcal_R/R$ to $K$.  
Since $X$ is smooth, we conclude that $\Xcal_{R,\eta}$ 
is smooth as well \cite[Cor.~17.7.3]{EGA4}.
By Proposition \ref{regulari-prop2}(vii), it is enough to prove the claim for any positive multiple of $\theta$.  
Using this and Lemma \ref{higher geometric model} below, we see that by passing to dominant models, we may assume that $\Xcal_R$ is {\it regular} and that 
the restriction of $\Lcal_R$ to  $\Xcal_{R,\eta}$ extends to an ample line bundle on $\Xcal_R$. 
By Remark \ref{equivalent geometric assumption}, these conditions are equivalent to Assumption \ref{Geometric assumption} and hence
 the claim follows from Theorem \ref{relative Theorem 8.5}.
\end{proof}

\begin{lemma} \label{higher geometric model}
Let $R$ be a discrete valuation ring which is defined geometrically by a $d$-dimensional normal variety $B$ over the field $k$
(as in Remark \ref{equivalent geometric assumption}). 
Let $\Xcal_R$ be a projective integral scheme over $R$ 
with $n$-dimensional regular generic fiber $X' \coloneqq \Xcal_{R,\eta}$. 
We assume that resolution of singularities holds over $k$ 
in dimension $d+n$. 
Then for any ample line bundle $L'$ on $X'$, there 
exists $m\in \N_{>0}$ and an ample extension 
$\Lcal_R'$ of $(L')^{\otimes m}$ to a regular  
$R$-model $\Xcal_R'$ of $X'$ with a projective 
morphism $\Xcal_R' \to \Xcal_R$ over $R$ extending 
the identity on $X'$.
\end{lemma}

\begin{proof}
The proof proceeds in three steps. 
First, we use a result of L\"utkebohmert about 
vertical blowing ups to show that $L'$ may be assumed 
to extend to an ample line bundle $\Hcal$ on $\Xcal_R$. 
In a second step, we show that  $\Xcal_R$ may be  
also assumed to be semi-factorial by a theorem of P\'epin. 
In a third step, we use resolution of singularities to 
construct our desired regular model $\Xcal_R'$.

\emph{Step 1:} 
Replacing $L'$ by a positive tensor power, 
we may assume that $L'$ has an ample extension $\Hcal_R$ to a projective $R$-model   $\KY_R$. 
There is a blow up $\pi:\mathscr Z_R \to \KY_R$ in an ideal sheaf $\Jcal$ supported in the special fiber of $\KY_R$ 
such that the identity on $X'$ extends to a morphism $\mathscr Z_R \to \Xcal_R$ \cite[Lemma 2.2]{luetkebohmert1993}. 
Then $\pi^{-1}(\Jcal)=\Ocal_{\mathscr Z_R/\KY_R}(1)$ and 
hence there is $\ell \in \N_{>0}$ such that 
$\pi^*(\Hcal^{\otimes \ell}) \otimes \Ocal_{\mathscr Z_R/\KY_R}(1)$ is ample \cite[Prop.~II.7.10]{Hart}. 
We conclude that by replacing $\Xcal_R$ by $\mathscr Z_R$ and by passing to a positive tensor power of $L'$, 
we may assume that $L'$ has an ample extension $\Hcal_R$ to $\Xcal_R$. This completes the first step.

\emph{Step 2:} 
By a result of P\'epin \cite[Thm.~3.1]{pepin2013}, 
there is a a blowing-up morphism $\pi' \colon \mathscr Z_R' \to \Xcal_R$ centered in the special fiber of $\Xcal_R$ 
such that $\mathscr Z_R'$ is semi-factorial.
The latter means that every line bundle on the generic fiber $\mathscr Z_{R,\eta}'$ of $\mathscr Z_R'$ over $R$ 
extends to a line bundle on $\mathscr Z_R'$. Similarly as in the first step, 
we may assume that a positive tensor power of $L'$ extends to an ample line bundle on $\mathscr Z_R'$. 
Replacing $\Xcal_R$ by $\mathscr Z_R'$ and $L'$ by this positive tensor power, we get the second step.

\emph{Step 3:} 
We may assume that $B$ is affine. 
Using $R = \Ocal_{B,b}$ for some $b \in B^{(1)}$, it is clear that $\Xcal_R$ extends to a projective integral scheme $\Xcal_B$ over $B$. 
By using resolution of singularities over $k$ in dimension $d+n$, there is a regular integral scheme $\Xcal_B'$ and 
a projective morphism $\varphi_B \colon \Xcal_B' \to \Xcal_B$ which is an isomorphism over the regular locus of $\Xcal_B$. 
Since $X'$ is contained in the regular locus of $\Xcal_B$, 
we conclude that $\varphi_B$ maps the generic fiber $\Xcal_{B,\eta}'$ of $\Xcal_B'$ over $B$ isomorphically onto $X'=\Xcal_{B,\eta}$. 
As usual, we read this isomorphism as an identification. 
Then we get an induced projective morphism
$$\varphi_R:\Xcal_R' \coloneqq \Xcal_B' \times_{B} {\Spec (R)} \longrightarrow \Xcal_R$$ 
extending the identity on $X'$. 
The same argument as in the first step gives $m \in \N_{>0}$ such that 
$$\Lcal_R' \coloneqq \varphi_R^*(\Hcal_R^{\otimes m}) \otimes \Ocal_{\Xcal_R'/\Xcal_R}(1)$$
 is an ample line bundle on $\Xcal_R'$. 
Let $F$ be the restriction of $\Ocal_{\Xcal_R'/\Xcal_R}(1)$ to $\Xcal_{R,\eta}' = \Xcal_{R,\eta}=X'$. 
Then $\Lcal_R'$ is a model of $(L')^{\otimes m} \otimes F$. 
To prove the lemma, we have to ensure that  $F$ may be assumed to be $\Ocal_{X'}$. 
To do so, we use that $\Xcal_R$ is semi-factorial to extend $F$ to a line bundle $\Fcal_B$ on $\Xcal_R$. 
Then we may replace $\Ocal_{\Xcal_R'/\Xcal_R}(1)$ by  $\Ocal_{\Xcal_R'/\Xcal_R}(1) \otimes \varphi_R^*(\Fcal^{-1})$ to deduce the claim.
\end{proof}

\section{Continuity of the envelope}\label{section6}

We present consequences of our application of
test ideals in the last subsection under the assumption that we
have resolution of singularities.
As 
in Section \ref{section uniform convergence}
let $K$ be a complete discretely 
valued field of positive characteristic $p>0$.
Let $X$ be a smooth projective variety over $K$ of dimension $n$.
Consider $\theta \in \Zcal^{1,1}(X)$
with ample de Rham class $\{\theta\}\in N^1(X)$.

{
\begin{defi}
\label{definition geometric origin generic}\label{defi dagger}\label{geometric origin}
We say that $X$ is 
\emph{of geometric origin
from a $d$-dimensional family over a field $k$} if
there exist 
a normal $d$-dimensional variety $B$ over $k$, 
a point $b\in B^{(1)}$ of codimension one
and a projective variety $Y$ over $K'= k(B)$ such that
\begin{enumerate}
\item
there exists an isomorphism $\widehat\KO_{B,b}\stackrel{\sim}{\to} K^\circ$ 
of rings where $\widehat\KO_{B,b}$ denotes the completion of the
discrete valuation ring $R \coloneqq \KO_{B,b}$,
\item 
an isomorphism
$Y \otimes_{K'} K \simeq X$ over $K$ 
with $K' \to K$ induced by (i).
\end{enumerate}
Usually, we read these isomorphisms as identifications. 
Moreover, if $L$ is a line bundle (resp. if $\theta$ is a closed $(1,1)$-form) on $X$, we say that $(X,L)$ 
(resp. $(X,\theta)$) is 
\emph{of geometric origin
from a $d$-dimensional family over a field $k$} if  the above conditions are satisfied and if we can also find 
a line bundle $L'$ on $Y$ inducing $L$ by the base change  $K/K'$  
(resp. a line bundle $\Lcal_R$ on an $R$-model  $\Xcal_R$ of $Y$ inducing $\theta$ by the base change  $\kcirc/R$).
\end{defi}}

We can now formulate our main result about the 
continuity of the envelope:

\begin{theo}\label{thm-continuity-theta-envelope}
Let $X$ be a smooth $n$-dimensional projective variety over $K$ of geometric origin
from a $d$-dimensional family over a perfect field $k$.
Assume that resolution of singularities holds over $k$ 
in dimension $d+n$.
If $\theta$ is a closed $(1,1)$-form on $X$ with ample de Rham class $\{\theta\}$ and 
if $u \in C^0(\Xan)$, then $\pshenvelope_\theta(u)$ is 
a uniform limit of $\theta$-psh model functions and 
thus $\pshenvelope_\theta(u)$ is  continuous on $\Xan$. 
\end{theo}

Using resolution of singularities in dimension three over a perfect field proven by Cossart--Piltant (see Theorem \ref{thm-cp}), we get the following application:

\begin{cor} \label{continuity for surfaces}
Let $X$ be a smooth projective surface over $K$ of geometric origin from a $1$-dimensional family over a perfect field $k$. 
Then the conclusion of Theorem \ref{thm-continuity-theta-envelope} holds unconditionally.
\end{cor}
	
We will prove Theorem \ref{thm-continuity-theta-envelope} in several steps. First, we prove it in a completely geometric situation:

\begin{lemma}\label{thm-continuity-theta-envelope'}
If we assume additionally that 
$(X,\theta)$ is of geometric origin
from a $d$-dimensional family over a perfect field $k$, 
then Theorem \ref{thm-continuity-theta-envelope} holds.
\end{lemma}

\begin{proof}
Recall that the space of model functions $\modelD(X)$ is dense in $C^0(X^\an)$ for the
topology of uniform convergence \cite[Thm.~7.12]{gubler-crelle}.
Hence we may assume that $u\in \modelD(X)$
by Proposition \ref{regulari-prop2}(v). 
Observe that by Proposition \ref{regulari-prop2}(vii), we may 
replace $(\theta,u)$ by a suitable multiple.
Hence we may assume without loss of generality 
that the model function $u$ is defined by a vertical divisor on a  $\kcirc$-model $\Xcal'$. It is clear that we can choose $\Xcal'$ dominating 
the geometric model $\Xcal=\Xcal_R \otimes_R \kcirc$ of $\theta$ from Definition  \ref{geometric origin}. 
It follows from Proposition \ref{models are geometric}{(a)}
that we may assume  $\Xcal = \Xcal'$. By Proposition \ref{regulari-prop2}(iv) we get
\begin{equation} \label{change to 0}
\pshenvelope_\theta(u)=\pshenvelope_{\theta+dd^cu}(0)+u.
\end{equation}
By construction, the class $\theta + dd^cu$ is induced by a line bundle on $\Xcal_R$ and 
hence Corollary \ref{zero envelope} yields that 
$\pshenvelope_{\theta+dd^cu}(0)$ is a uniform limit of $(\theta + dd^c u)$-psh model functions 
$\varphi_i$. 
Then $P_\theta(u)$ is the uniform limit of the sequence 
of $\theta$-psh functions 
$\varphi_i+u$
 by \eqref{change to 0}. 
\end{proof}

In the lemma above we have proven Theorem \ref{thm-continuity-theta-envelope} 
under the additional assumption that the $(1,1)$-form $\theta$ is defined geometrically. 
In the next lemma, we relax this assumption a bit only assuming that the de Rham class of $\theta$ is defined geometrically.

\begin{lemma}
\label{cor continuity function field}
If we assume additionally that  the de Rham class $\{\theta\}$ is induced by an ample line bundle $L$ on $X$ 
such that $(X,L)$ is of geometric origin from a $d$-dimensional family over the perfect field $k$,  
then Theorem \ref{thm-continuity-theta-envelope} holds.
\end{lemma}

\begin{proof}
By Proposition \ref{regulari-prop2}(vi), we may assume that $\theta \in \Zcal^{1,1}(X)_\Q$. 
By Proposition \ref{regulari-prop2}(vii), 
we may replace $L$ by a positive tensor power and 
$\theta$ by the corresponding multiple, and so we may assume that $L$ is very ample. 
In the notation of Definition \ref{geometric origin}, 
the assumption that $(X,L)$ is of geometric origin means that $L$ is the pull-back of a line bundle $L'$ on the projective variety $Y$ over $K'=k(B)$. 
It follows easily from \cite[Prop.~III.9.3]{Hart} and 
\cite[Prop.~2.7.1(xii)]{EGA4} that $L'$ is very ample. 
Then $L'$ extends to a very ample line bundle on a projective $R$-model of $Y$ 
for the discrete valuation ring $R = \Ocal_{B,b}$ from Definition \ref{geometric origin}. 
By base change to $\kcirc$, we conclude that there is a closed $(1,1)$-form $\theta'$ on $\Xan$ 
with de Rham class $\{\theta'\}=\{\theta\} $ such that $(X,\theta')$ is of geometric origin from a $d$-dimensional family over $k$.

By the $dd^c$-lemma in \cite[Thm.~4.3]{BFJ1} 
(see also the second author's thesis \cite[Thm.~4.2.7]{JellThesis} for generalizations) and 
using the rationality assumption on $\theta$ from the beginning of the proof, 
there is $v \in \modelD(X) $ such that $\theta' = \theta +dd^cv$. It follows from Proposition \ref{regulari-prop2}(iv) that 
\begin{align*}
\pshenvelope_\theta(u)-v=\pshenvelope_{\theta+dd^cv}(u-v)=\pshenvelope_{\theta'}(u-v).
\end{align*}
By Lemma \ref{thm-continuity-theta-envelope'}, the function $\pshenvelope_{\theta'}(u-v)$ is a uniform limit of $\theta'$-psh functions. 
Adding $v$, we get the claim for $\pshenvelope_\theta(u)$.
\end{proof}

To prove Theorem \ref{thm-continuity-theta-envelope} in full generality, 
the idea is to reduce to the above geometric situation by a similar trick as in \cite[Appendix A]{BFJ2}.

\begin{proof}[Proof of Theorem \ref{thm-continuity-theta-envelope}]
We note first that by 
Proposition \ref{regulari-prop2} (viii)
the  property that $\pshenvelope_\theta(u)$ is a uniform limit of $\theta$-psh model functions
is equivalent to the property that it is a continuous function. 
Let $K' / K$ be a finite normal extension and denote by $q \colon X' \coloneqq X \otimes_K{K'} \to X$ the natural projection.
Let $\theta' = q^*\theta \in \Zcal^{1,1}(X')$.
Then by Lemma \ref{lemma A4 BFJ2} we have that 
\[
q^* \pshenvelope_\theta(u) = \pshenvelope_{\theta'}(q^*u).
\]
It follows from \cite[Prop.~1.3.5]{berkovich-book}  
that   $\Xan$ is as a topological space equal to the quotient of $(X')^{\rm an}$ by the automorphism group of $K'/K$. 
We conclude that $\pshenvelope_\theta(u)$ is continuous if and only if $\pshenvelope_{\theta'}(q^*u)$ is continuous.

Hence 
we can replace $K$ by a finite normal extension.
Adapting the same argument as in  \cite[Lemma A.7]{BFJ2} to characteristic $p$, 
there exists a finite normal extension $K'/K$ and 
a function field $F$ of transcendence degree $d$ over $k$ with $K'$ as completion as in Definiton \ref{definition geometric origin generic}  
such that $X' \coloneqq X \otimes_K K'$ is the base change of a projective variety $Y$ over $F$ with 
$N^1(Y /F)_\Q \to N^1(X'/K')_\Q$ 
surjective. 
Replacing $K'$ by $K$,  
we can assume that there is $Y$ as above with a surjective map
\begin{equation}
\label{eq galois trick}
N^1(Y /F)_\Q \to N^1(X/K)_\Q
\end{equation}
induced by the natural projection $X \to Y$.  
To prove continuity of $\pshenvelope_\theta(u)$, 
we may assume that the de Rham class $\{\theta\}$ is in $N^1(X)_\Q$ 
by using an approximation argument based on Proposition \ref{regulari-prop2}(vi).
We conclude from surjectivity in \eqref{eq galois trick}  
that there is a non-zero $m \in {\N}$ such that  $\{m \theta\}$ is induced by a line bundle $L$ 
with $(X,L)$ of geometric origin from a $d$-dimensional family over $k$ (in fact from $Y$).
By Proposition \ref{regulari-prop2}(vii), we have  
$\pshenvelope_{\theta}(u)=\frac{1}{m}\pshenvelope_{m\theta}(mu)$ 
and hence continuity follows from Lemma \ref{cor continuity function field}.
\end{proof}

\section{The Monge--Amp\`ere equation}
\label{section7}

Let $K$ be a field endowed with a complete discrete absolute value. Boucksom, Favre, and Jonsson have shown in \cite{BFJ1, BFJ2} 
that the Monge--Amp\`ere equation for a Radon measure supported on 
the skeleton of a smooth projective variety over $K$ has 
a solution if the variety is of geometric origin from a one-dimensional 
family over a field  of characteristic zero. 
In this section, we will explain that the same is true in 
characteristic $p > 0$ if we assume resolution of singularities {(see Section \ref{new-section-6} for precise definitions).}

In the following, we work under the following assumptions:

\begin{itemize}
\item[(A1)]  
The $n$-dimensional smooth projective variety $X$ 
over $K$ is of  geometric origin from a $d$-dimensional family over a 
perfect field $k$ of characteristic $p>0$.  
\item[(A2)] 
Resolution of singularities holds over $k$ in dimension $d+n$.
\item[(A3)] 
Embedded resolution of singularities holds over $k$ in dimension $d+n$. 
\end{itemize} 

{Note that assumptions (A2) and (A3) are unconditional for $n=2$ and $d=1$ by Theorem \ref{thm-cp} of Cossart and Piltant.  
For the following, it is crucial to 
have in mind that  models of $X$ can be defined geometrically which follows from Proposition 
\ref{models are geometric}.}

To transfer the results from \cite{BFJ1, BFJ2}, it is 
essential to note that every projective $\kcirc$-model 
of $X$ is dominated by a {projective} SNC-model of $X$. 
To see this, we note first that we may assume that 
the given  $\kcirc$-model is  of geometric origin 
over the perfect field $k$ by Proposition 
\ref{models are geometric}. 
Using resolution of singularities in dimension $d+n$ 
similarly as in the third step of the proof of Lemma \ref{higher geometric model}, 
we deduce that there is a regular {projective} scheme $\KX_R$ over 
$R$ as in Definition \ref{geometric origin} dominating the given model. 
Applying in the same way embedded resolution of singularities in 
dimension $d+n$ to the non-smooth fibers of $\KX_R$ 
over $R$, we may assume that the singular fibers of 
$\KX_R$ have the same support as a strict normal crossing divisor. 
Then base change to $\kcirc$ yields the claim as base change of the discrete valuation ring $\Ocal_{B,b}$ 
to its completion $\kcirc$ preserves regularity \cite[\href{http://stacks.math.columbia.edu/tag/0BG4}{Tag 0BG4}]{stacks-project} and 
strict normal crossing support.

Having sufficiently many {projective}  SNC-models of $X$ at hand, 
the density results of skeletons in $\Xan$ given in \cite[Sect.~3]{BFJ1} also hold in our case of equicharacteristic $p$.
Given a closed $(1,1)$-form $\theta$ with ample de Rham class $\{\theta\}$, 
the notion of $\theta$-psh functions on $\Xan$ introduced in \cite[Sect.~7]{BFJ1} keeps the same properties in our case. 
In fact, the results of \cite[Sect.~1--7]{BFJ1} and their proofs carry over to our setting. 

Note that we have already proven the continuity of the $\theta$-psh envelope in Theorem \ref{thm-continuity-theta-envelope}, 
which is the analogue of \cite[Thm.~8.3]{BFJ1}, by using test ideals instead of multiplier ideals. 
{For $u \in C^0(\Xan)$, we recall from \ref{definition envelope} that we have used a different definition of the $\theta$-psh envelope $\pshenvelope_\theta(u)$ than in \cite[Def.~8.1]{BFJ1}. 
{Both definitions agree in the equicharacteristic zero
situation by  \cite[Thm.~8.3 and Lemma 8.9]{BFJ1}.
If the characteristic of $K$ is positive and the Assumptions
(A1)--(A3) hold then we have explained above how to define 
$\theta$-psh functions.
We claim now that in this case both definitions of the envelope 
agree as well.}
Indeed, it follows from \cite[Lemma 8.4]{BFJ1} that the definitions agree on quasi-monomial points of $\Xan$. For any $x \in \Xan$, we consider the net $p_\Xcal(x)$ with $\Xcal$ ranging over all SNC models of $X$. By \cite[Cor.~3.9]{BFJ1}, this net of quasi-monomial points converges to $x$. It follows from continuity that the net $\pshenvelope_\theta(u)(p_\Xcal(x)) $ converges to $\pshenvelope_\theta(u)(x)$ for our definition of the envelope. By \cite[Thm.~7.11, Prop.~8.2(i)]{BFJ1}, the same convergence holds for their envelope and hence both definitions agree.}

This yields now in the same way as in \cite[Thm.~8.7]{BFJ1} that the following \emph{monotone regularization} holds:

\begin{cor} \label{regularization}
Under the assumptions (A1)--(A3), let $\theta$ be a closed $(1,1)$-form on $\Xan$ with ample de Rham class. 
Then  every $\theta$-psh function on $\Xan$ is the pointwise limit of a decreasing net of $\theta$-psh model functions on $\Xan$. 
\end{cor}

In \cite[Sect.~3]{BFJ2}, the monotone regularization is the basic ingredient to generalize 
the Monge--Amp\`ere operator from $\theta$-psh model functions to bounded $\theta$-psh functions and 
hence it applies also to our setting leading to the same results as in \cite[Sect.~3]{BFJ2}.
For a bounded $\theta$-psh function $\varphi$, we denote by $\MA_\theta(\varphi)$ the associated Monge-Amp\`ere measure on $X^\an$.
The definitions, results and arguments  from \cite[Sect.~4--6]{BFJ2} carry over without change. 
In particular, we may choose a decreasing sequence of $\theta$-psh model functions on $\Xan$ 
in the monotone regularization from Corollary \ref{regularization} similarly as in \cite[Prop.~4.7]{BFJ2}.

A crucial step is now to prove the following \emph{orthogonality property}:

\begin{theo} \label{orthogonality principle}
Under the assumptions (A1)--(A3), let $\theta$ be a closed $(1,1)$-form on $\Xan$ with ample de Rham class. 
Then for every continuous function $f$ on $\Xan$ with $\theta$-psh envelope $\pshenvelope_\theta(f)$, we have the orthogonality property
\begin{align*}
\int_{X^{\an}}(f-P_{\theta }(f))\MA_\theta(P_{\theta }(f))=0.
\end{align*}
\end{theo}

\begin{proof}
By Proposition \ref{regulari-prop2}(vi) and the continuity of the Monge--Amp\`ere measure given in \cite[Thm.~3.1]{BFJ2}, we may assume that $\theta \in \Zcal^{1,1}(X)_\Q$. Using Proposition \ref{regulari-prop2}(vii), we may assume that $\theta$ is induced by a line bundle of a model of $X$. Then the claim follows from \cite[Thm.~6.3.2]{BGJKM} as $\pshenvelope_\theta(f)$ is continuous by Theorem \ref{thm-continuity-theta-envelope}.
\end{proof}

As a consequence of the orthogonality property, 
we get differentiability of $E \circ \pshenvelope_\theta $ as in \cite[Thm.~7.2]{BFJ2} 
where $E$ is the energy from \cite[Sect.~6]{BFJ2}. 
We have now all ingredients available to solve the \emph{following Monge--Amp\`ere equation}. 

\begin{theo} \label{MA problem}
Under the assumptions (A1)--(A3), let $\theta$ be a closed $(1,1)$-form on $\Xan$ with ample de Rham class $\{\theta\}$ and 
let  $\mu$ be a positive Radon measure on $X^\an$ of mass $\{\theta\}^n$ supported on the skeleton of a {projective} SNC-model. 
Then there is a continuous $\theta$-psh function $\varphi$ on $\Xan$ 
such that $\MA_\theta(\varphi)=\mu$ and $\varphi$ is unique up to additive constants.
\end{theo}

\begin{proof}
It was shown in \cite[\S 8.1]{BFJ2} that uniqueness follows 
from a result of Yuan and Zhang in \cite{yuan-zhang}. 
To prove existence of a $\theta$-psh solution $\varphi$, 
we use the variational method of Boucksom, Favre, and Jonsson. 
The basic tools needed here are upper semicontinuity 
of the energy  \cite[Prop.~6.2]{BFJ2}, 
the compactness theorem \cite[Thm.~7.10]{BFJ1}, 
and the differentiability of $E \circ \pshenvelope_\theta$. 
As explained above, all these results are available in our setting. 
It remains to see that $\varphi$ is continuous and 
this is done by estimates in the spirit of Kolodziej 
as in \cite[\S 8.3]{BFJ2}. 
\end{proof}

\appendix

\section{The skeleton and the retraction in the toric case \\
by Jos\'e Ignacio Burgos Gil and Mart\'in Sombra}\label{appendix}

\setcounter{theo}{0}

In this appendix we give a combinatorial description of the skeleton
associated to a toric model of a toric variety, and of the
corresponding retraction.  We will use this description to show an
example of two models of the same variety that have the same
skeleton but different retractions. In turn this will 
give a counterexample to a higher dimensional extension of 
Proposition \ref{lemma psh retraction curve}. 

Let $(K,\absval)$ be a complete
non-archimedean discretely valued field, 
$K^{\circ}$ the valuation ring,  $k$ the residue field, and  $S=\Spec(K^{\circ})$. Let
$\varpi $ be a uniformizer of $K^{\circ}$ and write
\begin{displaymath}
\lambda _{K}=-\log |\varpi |.
\end{displaymath}

Let $X$ be a smooth projective variety over $K$ of dimension $n$ and
$X^{\an}$ the 
associated Berkovich analytic space.  Let $\Xcal$ be an SNC model of
$X$ over $S$, that is an SNC projective scheme $\KX$ over
$S$ with generic fiber $X$ such that the  special
fiber, which is not assumed to be reduced, agrees as a closed subset
with a simple normal crossing divisor $D$ of $\KX$. To the model
$\Xcal$ we can associate {a} skeleton  $\Delta_\Xcal\subset \Xan$ and a
retraction $p_{\Xcal}\colon X^{\an}\to \Delta 
_{\Xcal}$, see
\cite[\S 3]{BFJ1} for details.

Let $L$ be an ample line bundle on $X$ and
$\Lcal$ a
nef model of $L$ on $\Xcal$. Let 
$\theta$ be the semipositive $(1,1)$-form in the class of
$L$ corresponding to the model $\Lcal$. 
Let $\mu $ be a positive Radon measure on $X^{\an}$ with support in
$\Delta _{\Xcal}$ such that $\mu
(X^{\an})=\deg_{L}(X)$. The
Monge-Amp\`ere equation looks for a $\theta $-psh function $\varphi$
on $X^{\an}$ such that
\begin{equation}\label{eq:3}
  (dd^{c}\varphi+\theta )^{\wedge n}=\mu .
\end{equation}
With the generality we are discussing in this paragraph,
there is not yet a 
definition of the class of $\theta $-psh
functions with all the properties of classical pluripotential theory, but every 
good definition of this class should include the class of $\theta $-psh model
functions as introduced in \ref{max psh}.

The following question is natural and in case of being true would be
of great help to solve the Monge-Amp\`ere equation in positive and
mixed characteristic.

\begin{question}\label{question:1}
  With the previous hypotheses, is it true that any solution $\varphi$ to
  the Monge-Amp\`ere equation \eqref{eq:3} satisfies
  \begin{equation}\label{eq:2}
    \varphi=\varphi\circ p_{\Xcal}?
  \end{equation}
\end{question}

We will see that this question has a negative
answer by exhibiting a counterexample in the context of toric
varieties. In fact, that this question has a negative answer is
related with Proposition \ref{lemma psh retraction curve} not being
true in higher dimension.
To this aim, we will consider a
smooth projective variety $X$ of dimension 2 and two models $\Xcal$ and $\Xcal'$ that
have the same skeleton
\begin{displaymath}
  \Delta =\Delta _{\Xcal}=\Delta _{\Xcal'}
\end{displaymath}
but with different retractions $p_{\Xcal}\not = p_{\Xcal'}$. We will
fix a semipositive $(1,1)$-form $\theta $ that is realized in both models
and construct two model functions $\varphi$ and $\varphi'$ on
$X^{\an}$ satisfying
\begin{alignat}{2}
  \varphi&\not = \varphi', & \varphi\mid_{\Delta }&=\varphi'\mid_{\Delta },\\
  \varphi&=\varphi\circ p_{\Xcal},&
  \varphi'&=\varphi'\circ p_{\Xcal'}.
\end{alignat}
As a consequence of these properties, we deduce that $\varphi= \varphi'\circ
p_{\Xcal} $ and that $\varphi'= \varphi\circ
p_{\Xcal'} $. Moreover $\varphi'$ will be a $\theta $-psh model function while
the model function $\varphi$ will not be $\theta $-psh.
Let $\mu \coloneqq   (dd^{c}\varphi'+\theta )^{\wedge 2}$. This is a
positive 
measure with support on $\Delta $ and $\varphi'$ is a solution of
the corresponding Monge-Amp\`ere equation. Since
\begin{displaymath}
  \varphi'\not = \varphi = \varphi'\circ p_{\Xcal},
\end{displaymath}
we see that $\varphi'$ is a counterexample to Question \ref{question:1} for
the model $\Xcal$. Moreover, if Proposition \ref{lemma psh
  retraction curve} were true in dimension 2, then
$\varphi = \varphi'\circ p_{\Xcal}$ would be a $\theta $-psh model
function, but it is not.

We place ourselves in the framework and notation of \cite{BPS}. The
results below will make explicit the skeleton and the retraction
associated to a toric SNC model of a toric variety.

Let $\Tbb\simeq \Gbb_{\rm m}^{n}$ be a split torus over $K$. We denote by
\begin{displaymath}
  M=\Hom(\Tbb,\Gbb_{\rm m}),\quad N=\Hom(\Gbb_{\rm m},\Tbb),  
\end{displaymath}
the lattices of characters and one-parameter subgroups of $\Tbb$. Then
$M=N^{\vee}$. We also denote $N_{\Rbb}=N\otimes \Rbb$ and
$M_{\Rbb}=M\otimes \Rbb$. The pairing between $u\in N_{\Rbb}$ and
$x\in M_{\Rbb}$ is denoted  by $\langle x,u \rangle$. 

Let now $X$ be a proper toric variety over $K$ and $\Xcal$
a proper toric model of $X$ over $S$. Then $X$ is described by a complete fan
$\Sigma $ in $N_{\R}$ and $\Xcal$ is described by a complete
SCR-polyhedral complex $\Pi$ in $N_{\R}$ whose recession fan satisfies $\rec(\Pi
)=\Sigma $ \cite[Thm.~3.5.4]{BPS}.

There is a map $\zeta_{K} \colon N_{\Rbb} \to \Tbb^{\an}$
that sends $u\in N_{\Rbb}$ to the seminorm on $K [M]$ given by
\begin{displaymath}
  \left|\sum_{m\in M}\alpha _{m}\chi^{m}\right|=\max_{m}|\alpha
  _{m}|e^{-\lambda _{K}\langle m,u \rangle}. 
\end{displaymath}
This is a particular case of the map denoted by $\theta _{\sigma}$ in 
\cite[Prop.-Def.~4.2.12]{BPS} composed with the homothety
of ratio $\lambda _{K}$.

There is also a map $\Val_{K} \colon 
\Tbb^{\an}\to N_{\Rbb}$ that sends a point $p\in \Tbb^{\an}$ to the
point $\Val_{K}(p)\in N_{\Rbb}$ determined by
\begin{displaymath}
  \langle m,\Val_{K}(p) \rangle = -\frac{1}{\lambda _{K}}\log |\chi^{m}(p)|,
\end{displaymath}
see \cite[Sec.~4.1]{BPS}. From the definition, it follows that $\Val_{K}
\circ \,
\zeta_{K} =\Id_{N_{\Rbb}}$.

To each polyhedron $\Lambda \in \Pi $ there is associated an orbit
$O(\Lambda )$ for the action of $\Tbb_{k}$ on the special fiber $\Xcal_{s}$ \cite[\S
3.5]{BPS}.  We denote by $\xi_{\Lambda }$ the generic point of
$O(\Lambda )$.

The relation of $\Val_{K}$ with the reduction map is given
by \cite[Cor.~4.5.2]{BPS}: a point $p\in \Tbb^{\an}$
satisfies $\red(p)\in O(\Lambda )$ if and only if $\Val_{K}(p)\in
\rint(\Lambda )$. 
The relation of $\zeta _{K}$ with the reduction map is given
by the next result.

\begin{lemma}\label{lemm:1} Let $\Lambda \in \Pi $. If $u$ lies in the
  relative interior of  
  $\Lambda $, then $\red(\zeta
  _{K}(u))=\xi_{\Lambda }$.
\end{lemma}
\begin{proof}
  We use the notation of \cite[\S 3.5]{BPS}. {In particular, we write
  $\widetilde N=N\oplus \Z$ and $\widetilde M=M\oplus \Z$. Let $\sigma
  $ be the cone of $\widetilde N$ generated by the set
  $\{(x,1)\,|\, x\in \Lambda \}$ and write $\widetilde M_{\Lambda
  }=\widetilde M\cap \sigma ^{\vee}$, where $\sigma ^{\vee}$ is the
  dual cone of $\sigma $.} Let $\Xcal_{\Lambda }$ be
  the affine toric scheme associated to $\Lambda $. The ring of
  functions of $\Xcal_{\Lambda }$ is
  \begin{displaymath}
    K^{\circ}[\Xcal_{\Lambda }]=K^{\circ}[\widetilde M_{\Lambda
    }]/(\chi^{(0,1)}-\varpi ). 
  \end{displaymath}
The orbit $O(\Lambda )$ is a closed subscheme of $\Xcal_{\Lambda
}$. If $u\in \rint(\Lambda )$, the ideal of $O(\Lambda )$ is the
ideal generated by the 
monomials $\chi^{(m,l)}$ with $(m,l)\in \widetilde M_{\Lambda
}$ and $\langle m,u\rangle + l >0$.

The generic fiber of $\Xcal_{\Lambda }$ is the affine toric variety
$X_{\rec (\Lambda )}=\Spec(K[M_{\rec(\Lambda )}])$. The natural 
inclusion $K^{\circ}[\Xcal_{\Lambda }]\subset K[M_{\rec(\Lambda
  )}]$ is given by $\chi^{(m,l)}\mapsto \varpi ^{l}\chi^{m}$. Any
point $p\in X_{\rec(\Lambda )}^{\an}$ determines a 
seminorm on $K^{\circ}[\Xcal_{\Lambda }]$. The set of points of
$X^{\an}_{\rec(\Lambda )}$ whose reduction belongs to $\Xcal_{\Lambda
}$ is
\begin{displaymath}
  C=\{p\in X^{\an}_{\rec(\Lambda )}\mid |f(p)|\le 1, \ \forall f\in
  K^{\circ}[\Xcal_{\Lambda }]\}.  
\end{displaymath}
Given a point $p\in C$, then $\red(p)$ is the point corresponding to
the prime ideal
\begin{displaymath}
  \mathfrak{q}_{p}=\{f\in
  K^{\circ}[\Xcal_{\Lambda }]\mid  |f(p)|< 1 \}.
\end{displaymath}
Every $f\in K^{\circ}[\Xcal_{\Lambda }]$ can be written as a 
sum 
\begin{displaymath}
  f=\sum_{(m,l)\in \widetilde M_{\Lambda }} \alpha_{(m,l)} \chi^{(m,l)}, 
\end{displaymath}
with $|\alpha _{(m,l)}|=0, 1$ and 
only a finite number of coefficients $\alpha_{(m,l)}$ different
from zero.

By the definition of $\zeta_{K}$,
\begin{align*}
  \mathfrak{q}_{\zeta _{K}(u)}&=\left\{\sum \alpha _{(m,l) }\chi^{(m,l)}\in
                                K^{\circ}[\Xcal_{\Lambda }]\,\middle |\, |\varpi |^{l}
                                e^{-\lambda _{K}\langle m,u \rangle} <1
                                \right\}\\
                                &=\left\{\sum \alpha _{(m,l) }\chi^{(m,l)}\in
  K^{\circ}[\Xcal_{\Lambda }]\,\middle| \,\langle m,u \rangle + l >0
                                  \right\}.
\end{align*}
Since $u\in \rint(\Lambda )$, we deduce that $\mathfrak{q}_{\zeta
  _{K}(u)}$ is the ideal of $O(\Lambda )$ and therefore
\begin{math}
  \red(\zeta_{K}(u))=\xi_{\Lambda }.
\end{math}
\end{proof}

We add now to $X$ the condition of being regular, which is equivalent
to $\Sigma $ being
unimodular, and to $\Xcal$ the condition of being an SNC model. By
\cite[Chap.~IV, \S 3.I item  d)]{kempfetal}, $\Xcal$ is regular if
and only if the rational fan in $N_{\Rbb}\times \Rbb_{\ge 0}$
generated by $\Pi \times\{1\}$ is
unimodular. In this case, the model is always an SNC model. On the
other hand, by \cite[Example 3.6.11]{BPS2} the model will be strictly
semistable (SNC with reduced special fiber) if, in addition, all the
vertices are lattice points.

Since a unimodular fan is necessarily simplicial,
for each polyhedron $\Lambda
\in \Pi $, of dimension $t$, we can write
\begin{equation}\label{eq:1}
  \Lambda = \conv(p_{0},\dots, p_{s})+\cone(v_{s+1},\dots,v_{t}),
\end{equation}
where $s\le t$, $p_{i}$ are points of $N_\Rbb$ and $v_{i}$ are vectors
on the tangent space to $N_{\Rbb}$ at a point, that we identify with
$N_{\Rbb}$.

We define the  \emph{combinatorial skeleton} as
\begin{displaymath}
  \Delta _{\Pi} = \bigcup_{\substack{\Lambda  \in \Pi \\ \Lambda  \text{
        bounded}}} \Lambda  .
\end{displaymath}
There is a \emph{combinatorial retraction} $p_{\Pi}\colon N_{\Rbb}\to \Delta _{\Pi}$ defined
as follows. Let $u\in N_{\Rbb}$ and let $\Lambda  \in \Pi $ be a
polyhedron of dimension $t$
with $u\in \Lambda  $. Write $\Lambda  $ as in equation \eqref{eq:1}. Therefore
$u$ can be written uniquely as
\begin{equation}\label{eq:4}
  u=\sum_{i=0}^{s}a_{i} p_{i}+ \sum_{j=s+1}^{t}\lambda _{j} v_{j},
\end{equation}
with $a_{i},\lambda _{j}\ge 0$, and $\sum a_{i}=1$.
Then
\begin{displaymath}
  p_{\Pi }(u)=\sum_{i=0}^{s}a_{i} p_{i}.
\end{displaymath}

\begin{rem}
    The fan $\Sigma $ determines a compactification $N_{\Sigma }$ of
    $N_{\Rbb}$ as in \cite[\S 4.1]{BPS} such that the retraction
    $p_{\Pi }$ can be extended to a 
    continuous map $N_{\Sigma }\to \Delta _{\Pi }$. 
\end{rem}

The following result explicites the skeleton and retraction associated
to the model $\Xcal$.

\begin{theo} With the previous hypotheses, the skeleton $\Delta
  _{\Xcal}\subset X^{\an}$ is given by 
  \begin{displaymath}
    \Delta _{\Xcal}=\zeta_{K} (\Delta _{\Pi}).
  \end{displaymath}
  The restriction to $\Tbb^{\an}$ of the retraction $p_{\Xcal}$ is the
  composition
  \begin{displaymath}
    p_{\Xcal}\mid_{\Tbb^{\an}}=\zeta_{K} \circ p_{\Pi}\circ \Val_{K} .
  \end{displaymath}
\end{theo}
\begin{proof}
  We start by recalling the construction of $\Delta _{\Xcal}$ and
  $p_{\Xcal}$ from \cite{BFJ1}. Note that, in loc.~cit. the residue
  field $k$ is of characteristic zero, but once we assume that the
  model $\Xcal$ is an SNC model, using the results of 
  \cite[\S~3.1]{Mustata_Nicaise} it is possible to extend the
  presentation of \cite{BFJ1} to the
  case of positive and mixed characteristic.

  Let $\Div_{0}(\Xcal)$ be the group of vertical Cartier divisors on
  $\Xcal$. Denote $\Div_{0}(\Xcal)_{\Rbb}=\Div_{0}(\Xcal)\otimes \Rbb$
  and let $\Div_{0}(\Xcal)_{\Rbb}^{\ast}$ be the dual. As explained in 
  \ref{model functions},  each $D\in
  \Div_{0}(\Xcal)$ determines a model function  $\varphi_{D}$. The map
  $D\mapsto \varphi_{D}$ is linear in $D$ and can be extended by
  linearity  to a map
   $\Div_{0}(\Xcal)_{\Rbb}\to C^{0}(X^{\an})$. 

  There
  is a map $\ev_{\Xcal}\colon \Xan \to \Div_{0}(\Xcal)_{\Rbb}^{\ast}$
  determined by
  \begin{equation}\label{eq:8}
    \langle D,\ev_{\Xcal}(x) \rangle =\varphi_{D}(x).
  \end{equation}
  Let $D_{1},\dots,D_{\ell}$ be the components of the
  special fiber $\Xcal_{s}$. Each $D_{i}$, $i=1,\dots,\ell$, determines
  a divisorial point $x_{i}\in X^{\an}$ and we denote by
  $e_{i}=\ev_{\Xcal}(x_{i})$. For each $J\subset \{1,\dots,\ell\}$ we
  write $D_{J}=\bigcap_{j\in J}D_{j}$ and $\sigma
  _{J}=\conv(e_{j},j\in J)$.
  Then the abstract skeleton
  of $\Xcal$ is
  \begin{displaymath}
    \Delta ^{\abs}_{\Xcal}=\bigcup_{\substack{J\subset
        \{1,\dots,\ell\}\\D_{J}\not = \emptyset}} \sigma _{J}\subset
  \Div_{0}(\Xcal)_{\Rbb}^{\ast}.
  \end{displaymath}
By \cite[Thm.~3.1]{BFJ1}, the image of $\ev_{\Xcal}$ is $\Delta
^{\abs}_{\Xcal}$ and there exists a unique function
$\emb_{\Xcal}\colon \Delta ^{\abs}_{\Xcal}\to X^{\an}$ such that
\begin{enumerate}
\item \label{item:1} $\ev_{\Xcal}\circ \emb_{\Xcal}=\Id_{\Delta ^{\abs}_{\Xcal}}$;
\item \label{item:2} for each $s\in \Delta _{\Xcal}^{\abs}$, if $s\in \rint(\sigma
  _{J})$, then $\red(\emb_{\Xcal}(s))=\xi_{D_{J}}$, where
  $\xi_{D_{J}}$ is the generic point of $D_{J}$.
\end{enumerate}
Then the skeleton and the retraction are given by
\begin{displaymath}
 \Delta _{\Xcal}=\emb_{\Xcal}(\Delta ^{\abs}_{\Xcal})\quad \and \quad
 p_{\Xcal}=\emb_{\Xcal}\circ \ev_{\Xcal}. 
\end{displaymath}

We now go back to the regular toric case. In particular, $X$ is a toric smooth
projective variety
over $K$ and $\Xcal$ is a toric projective SNC model. Then all the
divisors of $\Div_{0}(\Xcal)$ are toric divisors. Therefore, for $D\in
\Div_{0}(\Xcal)_{\Rbb}$, the function $\varphi_{D}$ is invariant under
the action of the compact torus $\Sbb=\Val_{K}^{-1}(0)$. The restriction
of $\varphi_{D}$ to $\Tbb^{\an}$ factorizes as
\begin{equation}\label{eq:9}
  \varphi_{D}\mid _{\Tbb^{\an}}=-\phi _{D}\circ \Val_{K},
\end{equation}
where $\phi _{D}$ is the function from \cite[Def.~4.3.6]{BPS} corresponding to the
trivial line bundle $\Ocal_{X}$ with the metric determined by $D$ and
the section $1$.

We now define $\ev_{\Pi}\colon N_{\Rbb}\to \Div_{0}(\Xcal)^{\ast}$ by
  \begin{displaymath}
    \langle D,\ev_{\Pi}(u) \rangle =-\phi_{D}(u).
  \end{displaymath}
  By construction, the restriction of $\ev_{\Pi}$ to each polyhedron
  $\Lambda  \in \Pi $ is affine. Moreover, using \eqref{eq:8} and
  \eqref{eq:9} we deduce that
  \begin{equation}
    \label{eq:7}
    \ev_{\Xcal}\mid _{\Tbb^{\an}}=\ev_{\Pi}\circ \Val_{K}.
  \end{equation}

  As before let $D_{1},\dots,D_{\ell}$ be the components of the
  special fiber $\Xcal_{s}$ and $x_{i}$ the divisorial point
  determined by $D_{i}$. Then the set of vertices of $\Pi $ is
  $\Pi^{0}=\{u_{1},\dots, u_{\ell}\}$, where
  $u_{i}=\Val_{K}(x_{i})$. Therefore $\ev_{\Pi}(u_{i})=e_{i}$. Since
  $\ev_{\Pi}$ is affine in each polyhedron of $\Pi$ we deduce that the
  image of $\ev_{\Pi}$ is $\Delta ^{\abs}_{\Xcal}$ and that
  $\ev_{\Pi}$ determines a homeomorphism
  $\Delta _{\Pi}\to \Delta ^{\abs}_{\Xcal}$. We define
  $\emb_{\Pi}\colon \Delta ^{\abs}_{\Pi}\to N_{\Rbb} $ as the
  composition of the inverse of this homeomorphism with the inclusion
  $\Delta _{\Pi}\hookrightarrow N_{\R}$. Using equation \eqref{eq:7}
  and Lemma \ref{lemm:1} one can check that
  $\zeta_{K} \circ \emb_{\Pi}$ satisfies the conditions \eqref{item:1}
  and \eqref{item:2} that characterize $\emb_{\Xcal}$. Therefore
  \begin{equation}\label{eq:6} 
    \emb_{\Xcal}=\zeta_{K} \circ \emb_{\Pi}.
  \end{equation}
  We next claim that 
  $ p_{\Pi }=\emb_{\Pi}\circ \ev_{\Pi}. $ Indeed, for every $D\in
  \Div_{0}(\Xcal)$, since $D$ is a model of the trivial vector bundle,
  we know that $\rec(\phi _{D})$ is the zero function. Therefore, writing any 
  $u\in \Lambda \in \Pi $ is as in \eqref{eq:4}, one can show that 
  \begin{displaymath}
    \phi _{D}=\phi _{D}\circ p_{\Pi }.
  \end{displaymath}
  This implies that $\ev_{\Pi }=\ev_{\Pi}\circ \,p_{\Pi }$. By
  construction $\emb_{\Pi}\circ \ev_{\Pi}$ is the identity in the
  image of $p_{\Pi }$. Therefore
  \begin{equation}\label{eq:5}
    \emb_{\Pi}\circ \ev_{\Pi}=\emb_{\Pi}\circ \ev_{\Pi}\circ p_{\Pi}=p_{\Pi}.
  \end{equation}

Using equations \eqref{eq:5} \eqref{eq:6} and \eqref{eq:7} we deduce
that 
  \begin{displaymath}
    \Delta _{\Xcal}=\emb_{\Xcal}(\Delta ^{\abs}_{\Xcal})=\zeta_{K}
    (\emb_{\Pi}(\Delta ^{\abs}_{\Xcal}))=\zeta_{K} (\Delta _{\Pi})
  \end{displaymath}
  and
  \begin{displaymath}
    p_{\Xcal}|_{\Tbb^{\an}}=\emb_{\Xcal}\circ
    \ev_{\Xcal}|_{\Tbb^{\an}}=\zeta_{K} \circ \emb_{\Pi }\circ
    \ev_{\Pi
    }\circ \Val_{K}= \zeta_{K} \circ p_{\Pi} \circ \Val_{K}
  \end{displaymath}
  concluding the proof.
\end{proof}

\begin{figure}[!h]
  \centering
  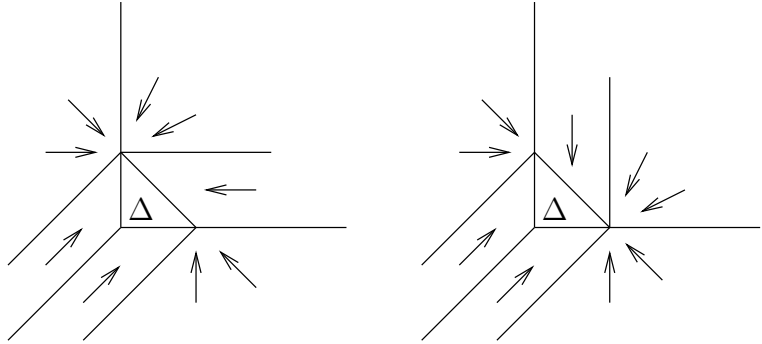  
  \caption{{Subdivisions corresponding to the toric models
      $\Xcal$ and $\Xcal'$.}
}  
  \label{fig:example}
\end{figure}

Consider the toric variety $X=\Pbb_K^2$. Let $\Pbb^2_{S}$ be the
projective space over $S$, 
and let $(x_0:x_1:x_2)$ be homogeneous coordinates of the special fiber
$\Pbb^2_{k}$. Consider the model $\Xcal$ of
$X$ obtained by blowing up 
$\Pbb^2_{S}$ at the line $x_2=0$ inside the special fiber and then
blowing up the strict transform   of the line $x_{1}=0$. The
SCR-polyhedral subdivision $\Pi $ associated to this model 
is depicted  in the left side of figure \ref{fig:example}.
Consider also the model $\Xcal'$ of
$X$ obtained as before, switching $x_{1}$ and $x_{2}$. The
SCR-polyhedral subdivision $\Pi' $ associated to this new model 
is depicted  in the right side of figure \ref{fig:example}. Both toric
schemes $\Xcal$ and $\Xcal'$ are SNC models (even more, they are
strictly semistable models) of $\Pbb^{2}_{K}$. 

The skeleton associated to both models is the simplex
\begin{displaymath}
  \Delta =\conv((0,0),(1,0),(0,1))
\end{displaymath}
and both retractions $p_{\Pi}$ and $p_{\Pi'}$ are also depicted on the
same figure. For 
instance the retraction $p_{\Pi}$ sends every point of $\sigma _{2}$ to the
point $(1,0)$, while the same retraction restricted to the polyhedron
$\sigma _{1}$ is the horizontal projection onto the segment
$\overline{(0,1)(1,0)}$ along the direction $(-1,0)$. By contrast, the
retraction $p_{\Pi'}$ sends both cones $\sigma _{1}'$ and $\sigma
_{2}'$ to the point $(1,0)$.

Consider the divisor $D$ of $\Pbb_K^{2}$ given by the line at infinity
and the divisor $\Dcal$ of $\Pbb_S^{2}$ given by the closure of $D$.
Let $L=\Ocal_{\Pbb_K^{2}}(D)$ and $\Lcal=\Ocal_{\Pbb^{2}_{S}}(\Dcal)$. Then $\Lcal $ is
a model of $L$ in $\Pbb_{S}^{2}$ and can be pulled back to both
$\Xcal$ and $\Xcal'$. Let $\theta $ be the
closed $(1,1)$-form defined by this model.

Let $\Psi \colon N_{\Rbb}\to \Rbb$ be the function
\begin{displaymath}
  \Psi (u,v)=\min(u,v,0).
\end{displaymath}
This is the function that determines the toric divisor $D$.

By \cite[Thm.~4.8.1]{BPS}, the space of all continuous $\theta
$-psh functions on $X^{\an}$ that are invariant under the action of
the compact torus $\Sbb$ can be identified with the set
of all bounded functions 
$f\colon N_{\Rbb}\to \Rbb$ such that $\Psi +f$ is concave.
This identification sends $f\colon N_{\Rbb}\to \Rbb$ to the
unique continuous function $\varphi\colon X^{\an}\to \Rbb$ such that  
$\varphi\mid _{\Tbb^{\an}}=-f\circ \Val_{K}$. 

Let $g\colon \Delta \to \Rbb$ the affine function that has the value
$1$ at the point $(1,0)$ and the value $0$ at the points $(0,0)$ and
$(0,1)$ and put
\begin{displaymath}
  f=g\circ p_{\Pi },\qquad f'= g\circ p_{\Pi '}.
\end{displaymath}
One easily verifies that
\begin{displaymath}
  \Psi +f'=\min(1,1+v,u)
\end{displaymath}
which is concave. On the other hand, the restriction of  $\Psi +f$ to
$\sigma _{3}$ is $0$ while its restriction to $\sigma _{1}$ is $1-v$,
hence $\Psi +f$ is not concave. 

Let $\varphi'$ be the continuous function on $X^{\an}$ whose
restriction to $\Tbb^{\an}$ is $-f'\circ \Val_{K} $. It is a
model $\theta $-psh function. The function $-f\circ \Val_{K} $ also extends to a model
function $\varphi$ on $X^{\an}$ but it is not $\theta$-psh because $f$ is not
concave \cite[Thm.~3.7.1~(2)]{BPS}.

We now write
\begin{displaymath}
  \mu =(dd^{c} \varphi' + \theta )^{\wedge 2}.
\end{displaymath}
By \cite[Thm.~4.7.4]{BPS} the measure $\mu $ is the atomic
measure with support in $\zeta_{K} ((1,0))$ with total mass one. Hence
its support is contained in $\Delta =\Delta _{\Xcal}$.

Summing up, $\mu $ is a measure with support in $\Delta =\Delta
_{\Xcal}$, the $\theta $-psh function $\varphi'$ is a solution of the
corresponding Monge-Amp\`ere equation but $\varphi'\not = \varphi'\circ p_{\Xcal}$
showing that the answer to Question \ref{question:1} is
negative. Moreover $\varphi=\varphi'\circ p_{\Xcal}$ is not $\theta $-psh, showing
that Proposition \ref{lemma psh retraction curve} does not extend to
dimension $\ge 2$.

Consider now a common refinement of the subdivisions corresponding to
the toric models $\Xcal$ and $\Xcal '$ and such that the
corresponding toric model  $\Xcal''$ is SNC. Then the skeleton $\Delta _{\Xcal''}$ will be strictly
bigger than $\Delta $. In particular $\Delta _{\Xcal''}$ contains the
unit square $0\le u,v\le 1$. It can be shown that
$\varphi'=\varphi'\circ p_{\Xcal''}$. Thus even if the solutions of
the Monge-Amp\`ere equation do not factor through the retraction
corresponding to the model $\Xcal$, they factor through the retraction
associated to a refined model. Mattias Jonsson asked us if one can
hope this phenomenon to hold in general. More concretely,  one can ask the
following question.

\begin{question}\label{question:2}
  With the same hypotheses as in Question \ref{question:1}, is it true that
  there exists a morphism of models $\Xcal'\to \Xcal$ such that 
  any solution $\varphi$ of
  the Monge-Amp\`ere equation \eqref{eq:3} satisfies
  \begin{equation}\label{eq:10}
    \varphi=\varphi\circ p_{\Xcal'}?
  \end{equation}
\end{question}

\medskip

{\footnotesize
{\sc J.I. Burgos Gil,
Instituto de Ciencias Matem\'aticas (CSIC-UAM-UCM-UCM3), 
Calle Nicol\'as Cabrera 15, Campus de la Universidad 
Aut\'onoma de Madrid, Cantoblanco, 28049 Madrid, Spain\\
\indent
{\it E-mail address:} {\tt burgos@icmat.es}

\medskip

M. Sombra, Instituci\'o Catalana de
Recerca i Estudis Avançats (ICREA). Passeig Llu\'is Companys~23,
08010 Barcelona, Spain \\
Departament de
Matem\`atiques i Inform\`atica, Universitat de Barcelona (UB). Gran
Via 585, 08007 Bar\-ce\-lo\-na, Spain\\
{\it E-mail address:} {\tt sombra@ub.edu}
}}

%\bibliographystyle{alpha}
%\bibliography{bib-regularization}

\newcommand{\etalchar}[1]{$^{#1}$}
\def\cprime{$'$}

\end{document}